\documentclass[11pt]{article}

\usepackage{xcolor}
\usepackage{amsfonts,amsmath,amssymb,mathtools,amsthm}
\usepackage{graphicx}
\usepackage{relsize}
\usepackage{stmaryrd}
\usepackage{natbib}
\usepackage{thmtools}

\usepackage{float}
\usepackage{mathtools} 

\def\br{\begin{remark}\rm\small}
\def\er{\end{remark}}
\def\bt{\begin{theorem}}
\def\et{\end{theorem}}
\def\bd{\begin{definition}}
\def\ed{\end{definition}}
\def\bp{\begin{proposition}}
\def\ep{\end{proposition}}
\def\bl{\begin{lemma}}
\def\el{\end{lemma}}
\def\bc{\begin{corollary}}
\def\ec{\end{corollary}}
\def\beaq{\begin{align}}
\def\eeaq{\end{align}}

\newcommand{\be}{\begin{equation}}
\newcommand{\ee}{\end{equation}}
\newcommand{\beq}{\begin{equation}}
\newcommand{\eeq}{\end{equation}}
\newcommand{\bea}{\begin{align}}
\newcommand{\eea}{\end{align}}
\newcommand{\beqq}{\begin{equation*}}
\newcommand{\eeqq}{\end{equation*}}
\newcommand{\beaa}{\begin{align*}}
\newcommand{\eeaa}{\end{align*}}

\newcommand{\trunc}{\mathrm{T}}

\newcommand{\Var}{{\operatorname{Var}}}

\newcommand{\vp}{{\mathrm{vp}}}

\newcommand{\td}{\tilde}
\theoremstyle{plain}
\newtheorem{theorem}{Theorem}[section]
\newtheorem{proposition}{Proposition}[section]
\newtheorem{lemma}[theorem]{Lemma}
\theoremstyle{remark}
\newtheorem{definition}[theorem]{Definition}

\usepackage[margin=1in]{geometry}
\usepackage[pdftex]{hyperref}
\hypersetup{colorlinks,citecolor=blue,linkcolor=blue,filecolor=black}

\def\abstractenglish{
This paper establishes the optimal sub-Gaussian variance proxy for truncated Gaussian and truncated exponential random variables. 
The proofs rely on first characterizing the optimal variance proxy as the unique solution to a set of two equations and then observing that for these two truncated distributions, one may find explicit solutions to this set of equations. Moreover, we establish the conditions under which the optimal variance proxy coincides with the variance, thereby characterizing the strict sub-Gaussianity of the truncated random variables. 
Specifically, we demonstrate that truncated Gaussian variables exhibit strict sub-Gaussian behavior if and only if they are symmetric, meaning their truncation is symmetric with respect to the mean. Conversely, truncated exponential variables are shown to never exhibit strict sub-Gaussian properties. These findings contribute to the understanding of these prevalent probability distributions in statistics and machine learning, providing a valuable foundation for improved and optimal modeling and decision-making processes. 
}

\def\titre{
Optimal sub-Gaussian variance proxy\\ for truncated Gaussian and exponential random variables
}

\def\support{The second author used part of his IUF junior grant G752IUFMAR for this research. The third author was partially supported by ANR-21-JSTM-0001 grant.}

\setcitestyle{citesep={,}}

\title{\titre}

\author{
Mathias Barreto\footnote{Higher School of Economics University, Moscow, Russia, \textsf{mbarretokonigliaro@edu.hse.ru}},
Olivier Marchal\footnote{Universit\'{e} Jean Monnet Saint-\'{E}tienne, CNRS, Institut Camille Jordan UMR 5208, Institut Universitaire de France, F-42023, Saint-\'{E}tienne, France, \textsf{olivier.marchal@univ-st-etienne.fr}},
and 
Julyan Arbel\footnote{Univ. Grenoble Alpes, Inria, CNRS, Grenoble INP, LJK, 38000 Grenoble, France, \textsf{julyan.arbel@inria.fr}}
}

\begin{document}

\maketitle

\begin{center}
    \textbf{Abstract}
\end{center}
\abstractenglish


\vspace{1.0cm}

\section{Introduction}

The sub-Gaussian property, a fundamental characteristic extensively explored in seminal works such as \cite{buldygin1980sub,boucheron2013concentration,pisier2016subgaussian}, has played a pivotal role in shaping the landscape of probability distributions. This property, defined by the tail behavior of random variables, has garnered significant attention for its implications in various mathematical disciplines such as 
concentration inequalities and large deviation estimates \citep{hoeffding1963probability,kearns1998large,ledoux1999concentration,bobkov1999exponential,raginsky2013concentration,boucheron2013concentration,berend2013concentration,perry2016statistical,ben2017concentration}, 
random series in relation to the geometry of Banach spaces \citep{Pisier1986ProbabilisticMI,Chow1966}, 
spectral properties of random matrices \citep{LITVAK2005491,RudelsonVershynin2009}, 
or Bayesian statistics \citep{elder2016bayesian,Catoni2018DimensionfreePB,vladimirova2019bayesianICML,Lee2020,vladimirova2021accurate}. More broadly, the sub-Gaussian property holds paramount importance in machine learning and artificial intelligence applications \citep{Devroye2016,ML4,AI2,AI1,AI3,ML5,ML1,ML2,ML3}.

The sub-Gaussianity of a random variable is a key determinant of its concentration properties which can be defined as follows. 

\begin{definition}\label{def:sub-Gaussian}
    A scalar random variable $X$ is called sub-Gaussian if there exists some $s\geq 0$ such that for all $\theta \in \mathbb{R}$:
    \begin{align}\label{eq:subGaussian-def}
       \mathbb{E}[\exp(\theta(X - \mathbb{E}[X]))] \leq \exp\left( \frac{s^2\theta^2}{2} \right).
    \end{align}
Any $s^2$ satisfying Equation~\eqref{eq:subGaussian-def} is called a variance proxy, and the smallest such $s^2$ is called the \textbf{optimal variance proxy}, which shall be denoted as $\|X\|_{\vp}^2$. It is well known that 
$$\Var[X]\leq\|X\|_{\vp}^2,$$ 
and a random variable satisfying $\Var[X]=\|X\|_{\vp}^2$ is called  \emph{strictly} sub-Gaussian.
\end{definition}
%
Note that there exist many equivalent ways for defining this sub-Gaussian property, each can be more useful depending on the sought application \citep[see Proposition 2.5.2 in][]{vershynin2018high}. The one recalled in Equation~\eqref{eq:subGaussian-def} is often referred to as the Laplace condition. It is equivalent to the following condition on the tails of $X$:
\begin{equation*}
    \exists\, C>0\,\, \text{such that}\,\,\forall\, t\geq0,\,\, \mathrm{P}(\vert X\vert \geq t)\leq 2e^{-\frac{t^2}{C^2}},
\end{equation*}
as well as the following condition on the moments of $X$:
\begin{equation*}
    \exists\, C>0\,\, \text{such that}\,\,\forall\, p\geq1,\,\, \mathbb{E}[\vert X\vert^p]\leq C^p p^{\frac{p}{2}}.
\end{equation*}
This shows that a sub-Gaussian random variable has finite moments of any order larger than one. 
\medskip

\noindent\textbf{Truncated random variables.}  
The focus of this paper lies on the relevance of the sub-Gaussian property to truncated Gaussian and truncated exponential distributions.  
Truncated distributions emerge as natural models when data exhibits inherent constraints or boundaries. 
These distributions find pervasive use in statistics and machine learning, where understanding and modeling the statistical properties of real-world phenomena are paramount. 

Notable fields include  
survival analysis, for handling censoring \citep{balakrishnan2000progressive}, 
reinforcement learning and bandit problems, for representing action probabilities or rewards that are bounded within certain limits \citep{bubeck2012regret,lattimore2020bandit,szepesvari2022algorithms},
Bayesian statistics, to represent prior knowledge or beliefs about parameters that are restricted to certain intervals \citep{gelman2014bayesian},
or sampling procedures, to reduce Monte Carlo error \citep{ionides2008truncatedIS,wawrzynski2007truncated}.
More broadly, handling truncation in data is useful in countless applications, ranging from clinical trials, financial modeling, to environmental modeling, underscoring the ubiquitous nature of truncated distributions in contemporary research and applications.\medskip

\noindent\textbf{Contributions.} In general, establishing that a random variable is sub-Gaussian might be easy, but complication stems from obtaining the optimal variance proxy. This has been achieved for the most commonly used distributions, including the beta and multinomial distributions \citep{marchal2017sub}, Bernoulli, binomial, Kumaraswamy, and triangular distributions \citep{arbel2020strict}. 
In the present work, we focus on establishing the optimal variance proxy for truncated Gaussian and truncated exponential distributions.
The proofs are based initially on defining the optimal variance proxy as the unique solutions to a pair of equations. Subsequently, it is observed that explicit solutions to this set of equations can be found for the two truncated distributions in question. 

Furthermore, we establish the conditions under which the optimal variance proxy matches the variance, thereby identifying the strict sub-Gaussianity of the truncated random variables.
In detail, we illustrate that truncated Gaussian variables display strict sub-Gaussian behavior if and only if they exhibit symmetry, meaning their truncation is symmetric relative to the mean. Conversely, truncated exponential variables are demonstrated to lack strict sub-Gaussian properties consistently. 
These results enhance our comprehension of these common probability distributions in statistics and machine learning, laying a solid groundwork for enhanced and optimal modeling and decision-making processes.
\medskip

\noindent\textbf{Outline.} 
We present the optimal sub-Gaussian variance proxy and the strict sub-Gaussianity results for truncated Gaussian and truncated exponential random variables respectively in Section~\ref{sec:gauss} and in Section~\ref{sec:expo}, along with the main proofs. 
Future research directions are proposed in Section~\ref{sec:directions}. Technical details on the proofs are deferred to Appendix~\ref{sec:app:gauss} and Appendix~\ref{sec:app:expo}, respectively for truncated Gaussian and truncated exponential random variables.


\section{Truncated normal random variables}\label{sec:gauss}

This section establishes the optimal variance proxy for a truncated normal variable. Observe that we \emph{a priori} know already that any truncated random variable is sub-Gaussian by Hoeffding's Lemma (when it is truncated along a finite interval). 

In general, if $X$ is a random variable with density $f_X$ and cumulative distribution function $F_X$ then its truncated version (which we shall denote by $X_\trunc$) inside the interval $(a,b)$ (for $a<b$) has the form:

\begin{align*}
    f_{X_\trunc}(x) = \begin{cases} \frac{f_X(x)}{F_X(b) - F_X(a)} \quad & \text{for $x \in (a,b)$}, \\
    0 \quad &\text{otherwise.}\end{cases}
\end{align*}

Let $X \sim \mathcal{N}(\mu,\sigma^2)$ and $\phi(\cdot), \Phi(\cdot)$ be the density and cumulative density of a standard normal variable. If we truncate $X$ on $(a,b)$ then it is known that its moment generating function, mean and variance are given by:

\begin{equation}
     \label{eq3.1}
\begin{split}
        \mathbb{E}\left[e^{\theta X_\trunc}\right] &= \exp\left\{ \theta \mu  + \theta^2 \frac{\sigma^2}{2}\right\}\left[\frac{\Phi(\beta - \sigma\theta)-\Phi(\alpha - \sigma\theta)}{\Phi(\beta)-\Phi(\alpha)}\right],\\
    \mathbb{E}[X_\trunc] &= \mu + \frac{\phi(\alpha)-\phi(\beta)}{\Phi(\beta)-\Phi(\alpha)}\sigma,  \\
    \Var[X_\trunc] &= \sigma^2\left[1-\frac{\beta\phi(\beta)-\alpha\phi(\alpha)}{\Phi(\beta)-\Phi(\alpha)}
-\left(\frac{\phi(\alpha)-\phi(\beta)}{\Phi(\beta)-\Phi(\alpha)}\right)^2\right] ,
\end{split}
\end{equation}
where 
\begin{equation*}
    \alpha\coloneqq \frac{a-\mu}{\sigma}\,\,,\,\, \beta\coloneqq \frac{b-\mu}{\sigma}.
\end{equation*}

Our main result establishes the optimal variance proxy for this density, which turns out to have a closed-form expression:

\begin{theorem}\label{thm:main_trunc_normal_non_standard}
   Let $X_\trunc$ be a normal variable with mean $\mu \in \mathbb{R}$ and variance $\sigma^2 \in \mathbb{R}_{>0}$ truncated along the interval $(a,b)$ with $a<b$. Then its optimal variance proxy is given by:
    \begin{align*}
        \|X_\trunc\|_{\vp}^2 = \sigma^2\times
        \begin{cases}
            1-\frac{2 \sigma}{b+a-2\mu}
            \frac{\phi(\frac{a-\mu}{\sigma})-\phi(\frac{b-\mu}{\sigma})}
            {\Phi(\frac{b-\mu}{\sigma})-\Phi(\frac{a-\mu}{\sigma})}
            \quad & \text{if}\,\, -\infty<a< b<+\infty \,\,\text{and}\,\,  a+b \neq 2\mu, \\
            1-
            \frac{2(b-\mu)}{\sigma}
            \frac{\phi(\frac{b-\mu}{\sigma})}{2\Phi(\frac{b-\mu}{\sigma})-1} \quad & \text{if}\,\, -\infty<a< b<+\infty \,\,\text{and}\,\, a+b = 2\mu,\\
            1 & \text{if}\,\, a=-\infty \,\,\text{and/or} \,\,b=+\infty.
        \end{cases} 
    \end{align*}
    In particular, $X_\trunc$ is \emph{strictly} sub-Gaussian if and only if $a+b=2\mu$, i.e. if and only if the truncation is symmetric relative to the mean $\mu$ of the Gaussian variable. 
\end{theorem}

In order to prove Theorem~\ref{thm:main_trunc_normal_non_standard}, let us first note that we can reduce the problem to the one of truncating a standard Gaussian random variable. 
Consider the transformation $Y\coloneqq \frac{X-\mu}{\sigma}$ and let $Y_\trunc$ be the truncated standard normal along the interval $(\alpha,\beta)$. Then, we have the relation:

\begin{align}
    \mathbb{E}\big[e^{\theta\left(X_\trunc-\mathbb{E}[X_\trunc]\right)}\big] &= \int_a^b e^{\theta(x-\mathbb{E}[X_\trunc])}\frac{\phi(\frac{x-\mu}{\sigma})}{\Phi(\frac{b-\mu}{\sigma})-\Phi(\frac{a-\mu}{\sigma})} \frac{dx}{\sigma} \notag\\
    &\overset{y=\frac{x-\mu}{\sigma}}{=} \int_{\alpha}^{\beta} \exp\left[\theta(y\sigma + \mu - \mathbb{E}[X_\trunc]) \right]\frac{\phi(y)}{\Phi(\beta)-\Phi(\alpha)}dy \notag\\
    &= \mathbb{E}\big[e^{\sigma\theta\left(Y_\trunc-\mathbb{E}[Y_\trunc]\right)}\big]. \label{eq3.2}
\end{align}

That is, to optimally bound the centered moment generating function of $X_\trunc$, we can restrict to optimally bounding that of $Y_\trunc$, as (\ref{eq3.2}) implies $\|X_\trunc\|_{\vp}^2 = \sigma^2 \|Y_\trunc\|_{\vp}^2$. Hence, an equivalent reformulation of Theorem~\ref{thm:main_trunc_normal_non_standard} is the following.

\begin{theorem}\label{thm:main_trunc_normal_standard}
   Let $Y_\trunc$ be a standard normal variable truncated in the interval $(\alpha,\beta)$ with $\alpha<\beta$. Then its optimal variance proxy is given by:
    \begin{align*}
        \|Y_\trunc\|_{\vp}^2 = 
        \begin{cases}
            1-\frac{2\left(\phi(\alpha)-\phi(\beta)\right)}{(\alpha+\beta)\left(\Phi(\beta)-\Phi(\alpha)\right)} \quad & \text{if}\,\, -\infty<\alpha< \beta<+\infty \,\,\text{and}\,\, \alpha \neq - \beta ,\\
            1-\frac{2\beta\phi(\beta)}{2\Phi(\beta)-1} \quad &
            \text{if}\,\, -\infty<\alpha< \beta<+\infty \,\,\text{and}\,\, \alpha = - \beta,\\
            1 & \text{if}\,\, \alpha=-\infty \,\,\text{and/or} \,\,\beta=+\infty.
        \end{cases} 
    \end{align*}
\end{theorem}
Theorem \ref{thm:main_trunc_normal_non_standard} and Theorem \ref{thm:main_trunc_normal_standard} are illustrated in Figure~\ref{fig:illustration-gaussian}. Note that only the case $\beta=-\alpha$ (or, equivalently, $a+b=2\mu$), yields strictly sub-Gaussian random variables. In this case, strict sub-Gaussianity is equivalent to symmetry (with respect to the mode/mean of the original Gaussian distribution). The relationship between symmetry and strict sub-Gaussianity is studied in \cite{arbel2020strict}.

\begin{figure}
    \centering
    \includegraphics[width=.7\textwidth]{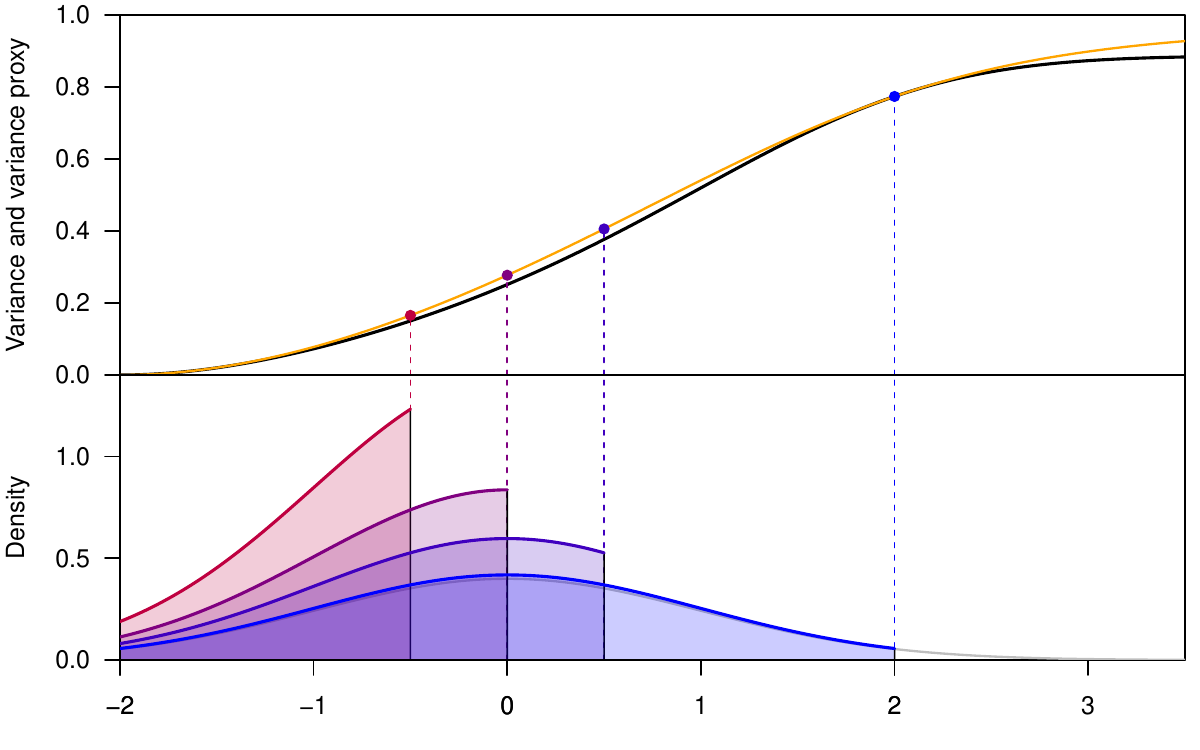}
    \caption{\textbf{Illustration of Theorem \ref{thm:main_trunc_normal_non_standard} and Theorem \ref{thm:main_trunc_normal_standard}.} The top panel represents the variance (black curve) and variance proxy (orange curve) of truncated standard Gaussian variables on intervals $(\alpha,\beta)$ with a fixed value of $\alpha=-2$ and varying values for $\beta\in\{-0.5, 0, 0.5, 2\}$, with colors from red to blue. The corresponding distributions are depicted on the bottom panel.  Note that only the case $\beta=2$ yields a symmetric distribution, which turns out to be strictly sub-Gaussian. See the blue dot on the right of the top panel where the variance and variance proxy are equal.}
    \label{fig:illustration-gaussian}
\end{figure}

\begin{proof}[Proof of Theorem \ref{thm:main_trunc_normal_standard}] 
Recall by Definition \ref{def:sub-Gaussian} that the optimal variance proxy of $Y_\trunc$ corresponds to the smallest possible $s^2 \geq 0$ such that 
\begin{equation*}
    \ln\left[ \exp\left(-\theta\left(\frac{\phi(\alpha)-\phi(\beta)}{\Phi(\beta)-\Phi(\alpha)}\right)+\frac{\theta^2}{2}\right)\left(\frac{\Phi(\beta-\theta)-\Phi(\alpha-\theta)}{\Phi(\beta)-\Phi(\alpha)} \right)  \right] \leq  \frac{s^2\theta^2}{2} \,\,,\,\, \forall \,\theta \in \mathbb{R}.
\end{equation*}
By defining:
\begin{gather*}
    c_{\alpha,\beta}\coloneqq \frac{\phi(\alpha)-\phi(\beta)}{\Phi(\beta)-\Phi(\alpha)}, \quad \theta_0 \coloneqq  \frac{\alpha+\beta}{2} \notag \\
    F_{\alpha,\beta}(\theta) \coloneqq  \Phi(\beta - \theta) - \Phi(\alpha - \theta), \quad f_{\alpha,\beta}(\theta) \coloneqq  \ln\left[ \frac{F_{\alpha,\beta}(\theta)}{\Phi(\beta)-\Phi(\alpha)} \right]\\
     h_{\alpha,\beta}(\theta)\coloneqq f_{\alpha,\beta}'(\theta) = \frac{\phi(\alpha-\theta)-\phi(\beta-\theta)}{\Phi(\beta-\theta)-\Phi(\alpha-\theta)}\\
    h_{\alpha,\beta}'(\theta) = \frac{\phi(\alpha-\theta)(\alpha-\theta)-\phi(\beta-\theta)(\beta-\theta)}{\Phi(\beta-\theta)-\Phi(\alpha-\theta)} - \left[\frac{\phi(\alpha-\theta)-\phi(\beta-\theta)}{\Phi(\beta-\theta)-\Phi(\alpha-\theta)}\right]^2, 
\end{gather*} 
this is equivalent to finding the smallest $w\coloneqq \frac{s^2-1}{2}\geq -\frac{1}{2}$ such that:
\beq  p_{\alpha,\beta;w}(\theta)\coloneqq w\theta^2 + c_{\alpha,\beta} \theta  \geq f_{\alpha,\beta}(\theta) \,\,,\,\, \forall\, \theta \in \mathbb{R}.
 \label{eq3.3}
\eeq

The delicate thing is understanding the right-hand side's function behavior. Note that this function is independent of the value of $w$. Let us first start with a lemma regarding the symmetry of the polynomial $p_{\alpha,\beta;w}$ and functions $f_{\alpha,\beta}$ and $h_{\alpha,\beta}$.

\begin{lemma}\label{Lemma3.2.1}
    For all $\theta \in \mathbb{R}$ it holds that:
    \begin{gather}
        p_{\alpha,\beta;w}\left(-\frac{c_{\alpha,\beta}}{2w}+\theta \right) = p_{\alpha,\beta;w}\left(-\frac{c_{\alpha,\beta}}{2w}-\theta\right), \quad f_{\alpha,\beta}\left(\theta_0-\theta\right) = f_{\alpha,\beta}\left(\theta_0+\theta\right), \notag \\
        h_{\alpha,\beta}\left(\theta_0-\theta\right) = -h_{\alpha,\beta}\left(\theta_0+\theta\right). \notag
    \end{gather}
\end{lemma}
\begin{proof}[Proof of Lemma \ref{Lemma3.2.1}]
The proof is immediate by direct computations.
\end{proof}

In particular observe that $p_{\alpha,\beta:w}$ and $f_{\alpha,\beta}$ share the symmetry line $\theta=\theta_0$ only when $w=-\frac{c_{\alpha,\beta}}{2\theta_0}$ or when $c_{\alpha,\beta} = 0$. A second important point is to note that:
\begin{equation*}
    f_{\alpha,\beta}(0)=0=p_{\alpha,\beta;w}(0)\,\,,\,\, f_{\alpha,\beta}(2\theta_0)=0\,,\,f'_{\alpha,\beta}(0)=p'_{\alpha,\beta;w}(0). 
\end{equation*}

Then, the crucial technical point of the proof of the theorem is the following lemma.

\begin{lemma}\label{lem:h_strict_concave}
    The function $h_{\alpha,\beta}$ is strictly convex on $(-\infty,\theta_0)$ and strictly concave on $(\theta_0,+\infty)$ for all $\alpha < \beta$. 
\end{lemma}
\begin{proof}[Proof of Lemma \ref{lem:h_strict_concave}]
The proof follows from direct and technical computations and is detailed in Appendix~\ref{sec:app:gauss}.
\end{proof}

Using Lemma \ref{Lemma3.2.1} and Lemma \ref{lem:h_strict_concave}, we obtain the optimal variance proxy for a truncated standard normal random variable given in Theorem \ref{thm:main_trunc_normal_standard}. \smallskip

\noindent\textbf{Case $\alpha$ and $\beta$ finite.} Observe that $f_{\alpha,\beta}$ is strictly increasing on $(-\infty,\theta_0)$ and strictly decreasing on $(\theta_0,+\infty)$. Thus, it achieves its maximum at $\theta=\theta_0$. Using the strict concavity of $h_{\alpha,\beta}$ inside the interval $(\theta_0,+\infty)$ we propose a geometric proof where the optimal variance proxy case corresponds to fitting the parabola $\theta\mapsto p_{\alpha,\beta;w}(\theta)$ inside the function $f_{\alpha,\beta}$ as illustrated in Figure \ref{Figure3.2}. 
\begin{figure}
\centering
\includegraphics[width=.6\textwidth]{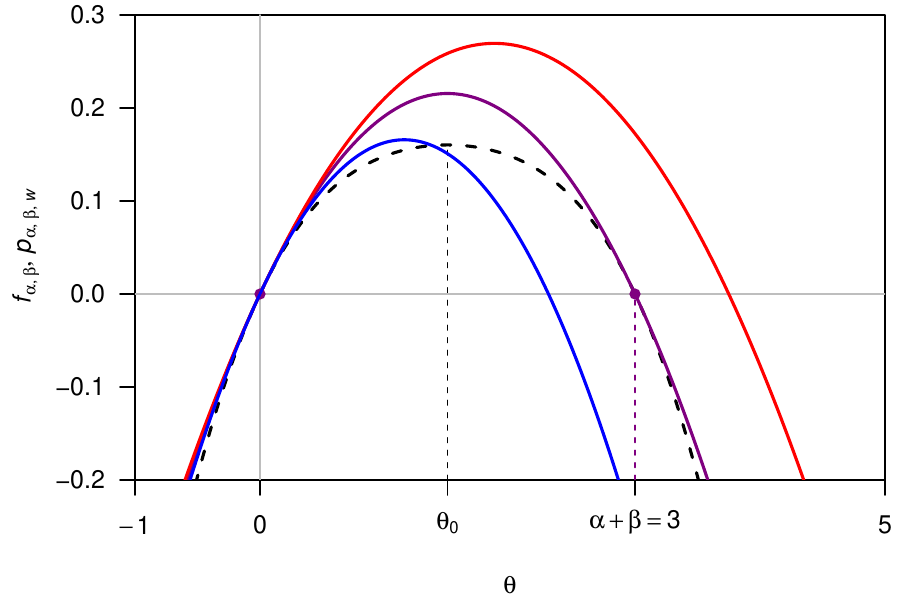}
\caption{\textbf{Illustration of the proof of Theorem~\ref{thm:main_trunc_normal_standard}} for $\alpha =-1$ and $\beta=4$, such that $\theta_0=(\alpha+\beta)/2=3/2$. The function $f_{\alpha,\beta}$ is plotted in dashed black line, while three polynomials $p_{\alpha,\beta;w}$  are plotted in color: purple corresponds to the use of the optimal variance proxy, where we can observe that $f_{\alpha,\beta}$ and $p_{\alpha,\beta;w_c}$ are tangent at the purple dot with coordinates $(\alpha+\beta,0)$. The red curve corresponds to a variance proxy that is not optimal, while the value used for the blue curve is not a variance proxy.
}
\label{Figure3.2}
\end{figure}

We now need to split into three cases.
\smallskip

\noindent\textbf{Case $\alpha$ and $\beta$ finite, $\alpha\neq -\beta$.} Let us consider the value $w_c=\frac{\phi(\beta)-\phi(\alpha)}{(\alpha+\beta)(\Phi(\beta)-\Phi(\alpha))}=-\frac{c_{\alpha,\beta}}{2\theta_0}$ corresponding to a parabola that is tangent to $f_{\alpha,\beta}$ at $(0,c_{\alpha,\beta})$ and $(2\theta_0,-c_{\alpha,\beta})$ and with a maximum at $\theta=\theta_0$.  Indeed, we have:
\begin{equation}\label{wcthreezeros} 
\begin{split}
f_{\alpha,\beta}(0)&=0=p_{\alpha,\beta;w_c}(0),\\
 f_{\alpha,\beta}(2\theta_0)&=0=4w_c\theta_0^2+c_{\alpha,\beta}\theta_0=p_{\alpha,\beta;w_c}(2\theta_0),\\
 f'_{\alpha,\beta}(\theta_0)&=0=2\theta_0w_c+c_{\alpha,\beta}=p'_{\alpha,\beta;w_c}(\theta_0),\\
f'_{\alpha,\beta}(2\theta_0)&=-c_{\alpha,\beta}=4\theta_0w_c+c_{\alpha,\beta}=p'_{\alpha,\beta;w_c}(2\theta_0),\\
f'_{\alpha,\beta}(0)&=c_{\alpha,\beta}=p'_{\alpha,\beta;w_c}(0).
\end{split}
\end{equation}

Let us define $f_{\alpha,\beta;w}\coloneqq f_{\alpha,\beta} -p_{\alpha,\beta;w}$ and $h_{\alpha,\beta;w}\coloneqq h_{\alpha,\beta} -p'_{\alpha,\beta;w}$ so that $h''_{\alpha,\beta;w}=h''_{\alpha,\beta}$ because $p_{\alpha,\beta;w}$ is a parabola. We will now prove that inequality \eqref{eq3.3} holds for $w=w_c$, i.e., we will check that $f_{\alpha,\beta;w_c} (\theta) \leq 0$ for all $\theta \in \mathbb{R}$. From Lemma~\ref{lem:h_strict_concave}, $h'_{\alpha,\beta;w}$ is strictly increasing on $(-\infty,\theta_0)$ and strictly decreasing on $(\theta_0,+\infty)$ and thus reaches its maximum at $\theta=\theta_0$. Thus, $h'_{\alpha,\beta;w}$ may have at most $2$ zeros and by Rolle's theorem, $h_{\alpha,\beta;w}$ may have at most $3$ zeros on $\mathbb{R}$. For $w=w_c$, \eqref{wcthreezeros} implies that $h_{\alpha,\beta;w_c}$ has three distinct zeros at $\theta\in\{0,\theta_0,2\theta_0\}$. This implies that $h'_{\alpha,\beta;w_c}(\theta_0)>0$ and that there exists $\theta_1<\theta_0<\theta_2$ such that $h'_{\alpha,\beta;w_c}(\theta_1)=h'_{\alpha,\beta;w_c}(\theta_2)=0$ and $h'_{\alpha,\beta;w_c}$ is strictly positive on $(\theta_1,\theta_2)$ and strictly negative on $(-\infty,\theta_1)\cup(\theta_2,+\infty)$ because of Lemma~\ref{lem:h_strict_concave}. Let us assume that $\theta_0>0$ (the case $\theta_0<0$ is similar but the set of zeros $\{0,\theta_0,2\theta_0\}$ is ordered in the opposite way) then variations of $h_{\alpha,\beta;w_c}$ imply that $0<\theta_1<\theta_0<\theta_2<2\theta_0$ and thus that $h_{\alpha,\beta;w_c}$ is strictly positive on $(-\infty,0)\cup(\theta_0,2\theta_0)$ and strictly negative on $(0,\theta_0)\cup(2\theta_0,+\infty)$. Hence $f_{\alpha,\beta;w_c}$ is strictly increasing on $(-\infty,0)\cup(\theta_0,2\theta_0)$ and strictly decreasing on $(0,\theta_0)\cup(2\theta_0,+\infty)$. Since, $f_{\alpha,\beta;w_c}(2\theta_0) = f_{\alpha,\beta;w_c}(0)=0$, we conclude that $f_{\alpha,\beta;w_c}$ is non-positive and hence that inequality \eqref{eq3.3} holds for $w=w_c$.\\

 Let us now prove that $w_c$ is optimal. Indeed, if we take by contradiction that $w<w_c$, then $f_{\alpha,\beta}(2\theta_0)=0$ while $p_{\alpha,\beta;w}(2\theta_0)=4w\theta_0^2+c_{\alpha,\beta}\theta_0<4w_c\theta_0^2+c_{\alpha,\beta}\theta_0=0$. Therefore inequality \eqref{eq3.3} is not realized in a neighborhood of $2\theta_0$. So $w_c$ is optimal.
\smallskip

\noindent\textbf{Case $\alpha$ and $\beta$ finite, $\alpha= -\beta$.} In this case, we have $c_{-\beta,\beta}=0$ and $\theta_0=0=2\theta_0$. Moreover, the two curves are always tangent at $\theta=0$ for any value of $w$:
\begin{align}
    \label{alphaminusbeta}f_{-\beta,\beta}(0)=0=p_{-\beta,\beta;w}(0) \quad\text{and}\quad
 f'_{\alpha,\beta}(0)=0=p'_{\alpha,\beta;w}(0).
\end{align}

Consider $w_c=-\frac{\beta\phi(\beta)}{2\Phi(\beta)-1}$ and let us verify that it is a variance proxy. In this case, the two curves are tangent at $\theta=0$ but the second derivatives are also the same:
\beq  \label{alphabeta} f''_{\alpha,\beta}(0)=-\frac{2\beta\phi(\beta)}{2\Phi(\beta)-1}=2w_c=p''_{\alpha,\beta;w_c}(0).\eeq
Let us define $f_{\alpha,\beta;w}\coloneqq f_{\alpha,\beta} -p_{\alpha,\beta;w}$ and $h_{\alpha,\beta;w}\coloneqq h_{\alpha,\beta} -p'_{\alpha,\beta;w}$ as in the previous case. From Lemma~\ref{lem:h_strict_concave}, $h'_{\alpha,\beta;w_c}$ is strictly increasing on $(-\infty,0)$ and strictly decreasing on $(0,+\infty)$. It achieves its maximum at $\theta=0$ and $h'_{\alpha,\beta;w_c}(0)=0$ from \eqref{alphabeta}. Thus $h'_{\alpha,\beta;w_c}$ is strictly negative on $\mathbb{R}\setminus\{0\}$ and $h_{\alpha,\beta;w_c}$ is a strictly decreasing function vanishing at $\theta=0$. Hence $f_{\alpha,\beta;w_c}$ is strictly increasing on $\mathbb{R}_{\leq 0}$ and  strictly decreasing on $\mathbb{R}_{\geq 0}$. Its maximum is thus achieved at $\theta=0$ and is null from \eqref{alphaminusbeta}. Thus, it remains negative on $\mathbb{R}$ and we conclude that inequality \eqref{eq3.3} holds for $w=w_c$.\\ 

Let us now prove that $w_c$ is optimal. Indeed, if we take by contradiction that $w<w_c$, then the two curves are still tangent at $\theta=0$ from \eqref{alphaminusbeta} but the second derivatives satisfy $f''_{\alpha,\beta}(0)=-\frac{2\beta\phi(\beta)}{2\Phi(\beta)-1}=2w_c>2w=p''_{\alpha,\beta;w}(0)$. Hence the parabola $p_{\alpha,\beta;w}$ is locally below the function $f_{\alpha,\beta}$ so that inequality \eqref{eq3.3} is not realized in a neighborhood of $\theta=0$. So $w_c$ is optimal.
\smallskip
In the end, using the fact that $s^2=2w+1$, we obtain that $s_c^2\coloneqq 2w_c+1$ is the optimal variance proxy.
\smallskip

\noindent\textbf{Case $\alpha=-\infty$ and/or $\beta=+\infty$.} Note first that the case when $\beta=-\alpha=+\infty$ boils down to not truncating the original standard Gaussian, which trivially implies strict sub-Gaussianity, while the case $\alpha$ finite and $\beta=+\infty$ is equivalent to the $\alpha=-\infty$ and $\beta$ finite by symmetry. Let us thus focus on the latter. Most of the results proved for arbitrary finite values of $\alpha$ extend by taking the limit $\alpha\to-\infty$. In particular, $h_{-\infty,\beta}$ is a strictly concave function on $\mathbb{R}$. Moreover, one can take the limit $\alpha\to -\infty$ in \eqref{eq3.3} since all quantities, including $w_c(\alpha)$, are continuous functions of $\alpha$. This provides that $s=1$ is a variance proxy (i.e. $w=0$) for the truncated Gaussian on $(-\infty,\beta)$. Let us now prove that this value is optimal. We observe that for $w<0$ we have:
\beq f_{-\infty,\beta;w}(\theta)\coloneqq f_{-\infty,\beta}(\theta) -p_{-\infty,\beta;w}(\theta)=\ln \left(\frac{\Phi(\beta-\theta)}{\Phi(\beta)}\right)-w\theta^2-c_{-\infty,\beta}\theta \overset{\theta\to -\infty}{\to} +\infty,\eeq so that \eqref{eq3.3} is obviously not verified in a neighborhood of $-\infty$. Hence $s=1$ is the optimal variance proxy when $\alpha=-\infty$. Similarly, $s=1$ is the optimal variance proxy for $\beta=+\infty$. This concludes the proof of Theorem~\ref{thm:main_trunc_normal_standard}. 
\end{proof}

\section{Truncated exponential random variables}\label{sec:expo}

This section considers the truncated version of the classical exponential distribution. Let $X\sim \text{Exp}(\lambda)$ be an exponential random variable with rate $\lambda$ (thus with mean $1/\lambda$), and let $X_\trunc$ denote the truncation of $X$ along the interval $(a,b)$ with $0 \leq a < b \leq +\infty$ (recall that the untruncated version is not sub-Gaussian). Its density, mean, and variance are of the form: 
%
\begin{align*}
    f_{X_{\trunc}}(t)&= \lambda  \frac{e^{-\lambda t}}{e^{-\lambda  a}-e^{-\lambda  b}} \boldsymbol{1}_{(a,b)}(t),\\
    \mathbb{E}[X_{\trunc}]&= \frac{1}{\lambda}+\frac{b e^{\lambda  a}-a e^{\lambda  b}}{e^{\lambda  a}-e^{\lambda b}}=\frac{1}{\lambda}+\frac{a e^{- \lambda a}-b e^{-\lambda b}}{e^{-\lambda a}-e^{-\lambda b}},\\
    \Var[X_{\trunc}]&= \frac{1}{\lambda^2}-\frac{(b-a)^2e^{\lambda(a+b)}}{\left(e^{\lambda b}-e^{\lambda a}\right)^2}.
\end{align*}

It is easy to see that $b$ must be finite in order for $X_\trunc$ to be sub-Gaussian. Indeed, we have:
\begin{equation}\mathbb{E}[\exp(\theta(X_{\trunc} - \mathbb{E}[X_{\trunc}]))]\big|_{\theta = \lambda} =\frac{\lambda(b-a)e^{-\lambda \mathbb{E}[X_{\trunc}]}}{e^{-\lambda a}-e^{-\lambda b}} \overset{\lambda\to +\infty}{\sim} e^{\lambda(a-\mathbb{E}[X_{\trunc}])} \lambda b,
\end{equation}
so that the inequality required in  Definition \ref{def:sub-Gaussian} at $\theta = \lambda$ cannot be realized for any $s^2\geq 0$ when $b\to +\infty$.

\begin{theorem}\label{thm:main_trunc_expo}
   Let $X_\trunc$ be an exponential variable with scale $\lambda$ truncated along the interval $(a,b)$ with $0\leq  a<b < \infty$. Then its optimal variance proxy is given by:
    \beqq\|X_\trunc\|_\vp^2=\frac{(b-a)\left(e^{\lambda b}+e^{\lambda a}\right)}{2\lambda\left(e^{\lambda b}-e^{\lambda a}\right)}-\frac{1}{\lambda^2}.\eeqq
    In particular, $X_\trunc$ is never strictly sub-Gaussian.
\end{theorem}

Similarly as in the Gaussian case, in order to prove Theorem~\ref{thm:main_trunc_expo} we restrict ourselves to the case of a standard exponential random variable $Y\sim \text{Exp}(1)$. Indeed, let $Y\sim  \text{Exp}(1)$ and $Y_\trunc$ be its truncation on $[\alpha,\beta]$ with $0\leq \alpha< \beta <+\infty$, then defining $X =  \frac{Y}{\lambda}$, $(a,b)=\left(\frac{\alpha}{\lambda}, \frac{\beta}{\lambda}\right)$ and $X_\trunc$ the truncation of $X$ on $[a,b]$, we have that $X_\trunc =  \frac{Y_\trunc}{\lambda}$. Thus, we get:

\begin{gather*} \mathbb{E}\left[\text{exp}\Big(\theta(X_{\trunc}-\mathbb{E}[X_{\trunc}]))\right]=\mathbb{E}\left[\text{exp}\left( \frac{\theta}{\lambda}(Y_{\trunc}-\mathbb{E}[Y_{\trunc}])\right)\right] \,\,,\,\, \forall\, \theta\in \mathbb{R},
\end{gather*}

from where it follows that $\lambda^2\|X_\trunc\|_\vp^2=\|Y_\trunc\|_\vp^2$. Hence an equivalent formulation is:

\begin{theorem}\label{thm:main_trunc_expo_standard}
   Let $Y_\trunc$ be an exponential variable with mean $1$ truncated along the interval $(\alpha,\beta)$ with $0\leq  \alpha<\beta < +\infty$. Then its variance proxy is given by:
    \beqq
    \|Y_\trunc\|_\vp^2=\frac{(\beta-\alpha)\left(e^{\alpha}+e^{\beta}\right)}{2\left(e^{\beta}-e^{\alpha}\right)}-1.
    \eeqq
    In particular, $Y_\trunc$ is \emph{never strictly} sub-Gaussian.
\end{theorem}

As discussed above, the exponential random variable obtained with $\beta=+\infty$ and any $\alpha$ is not sub-Gaussian. As a result, the variance proxy given in Theorem~\ref{thm:main_trunc_expo_standard} needs to compensate for this lack of sub-Gaussianity by diverging to $+\infty$ as $\beta\to+\infty$. More specifically, we have that $\|Y_\trunc\|_\vp^2$ is equivalent to $\beta/2$ as $\beta\to+\infty$, for fixed $\alpha$. 
Theorem \ref{thm:main_trunc_expo} and Theorem \ref{thm:main_trunc_expo_standard} are illustrated in Figure~\ref{fig:illustration-expo}.
\begin{figure}
    \centering
    \includegraphics[width=.7\textwidth]{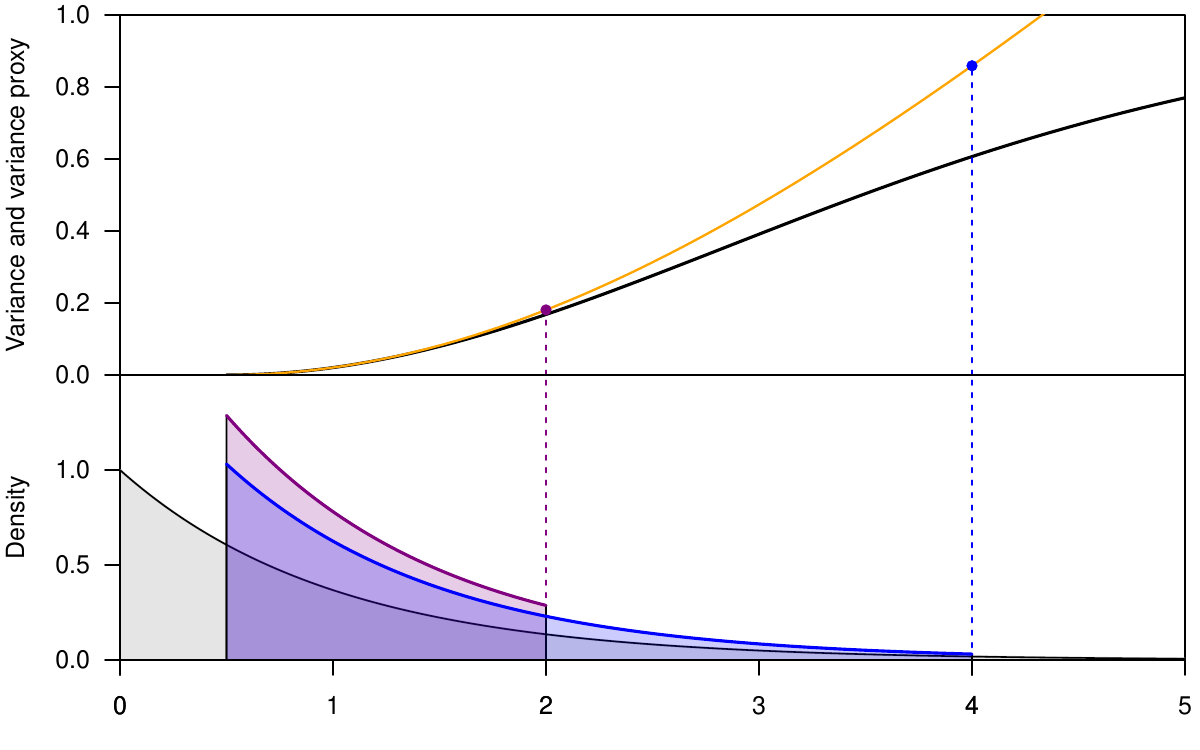}
    \caption{\textbf{Illustration of Theorem \ref{thm:main_trunc_expo} and Theorem \ref{thm:main_trunc_expo_standard}.}  The top panel represents the variance (black curve) and variance proxy (orange curve) of truncated standard exponential random variables on intervals $(\alpha,\beta)$ with a fixed value of $\alpha=1/2$ and varying values for $\beta\in\{2, 4\}$ (in violet and blue). The corresponding distributions are depicted on the bottom panel. Note that truncated standard exponential random variables are never strictly sub-Gaussian (the black and orange curves on the top panel never coincide).}
    \label{fig:illustration-expo}
\end{figure}

 The proof of Theorem \ref{thm:main_trunc_expo_standard} relies on a direct study of the difference
\begin{equation}
    \label{eq:g-definition}
    g_{\alpha,\beta,s}(\theta) \coloneqq \mathbb{E}[\exp(\theta(Y_\trunc - \mathbb{E}[Y_\trunc]))] - \exp\left( \frac{s^2\theta^2}{2} \right).
\end{equation} 

The most important step is to verify that the optimal variance proxy is characterized by the unique pair $(\theta_c, \|Y_{\trunc}\|_{\vp})$ that solves the following system of equations:

\begin{equation}\label{Eqexpotosatisfy}
    g_{\alpha,\beta,\|Y_{\trunc}\|_{\vp}}(\theta_c)=0 \,\, \text{ and } g'_{\alpha,\beta,\|Y_{\trunc}\|_{\vp}}(\theta_c)=0 \,\text{ with }s_{\text{inf}}(\alpha,\beta)<\|Y_{\trunc}\|_{\vp}\leq s_1(\alpha,\beta)\,\text{ and }\,\theta_c\neq 0,
\end{equation}

where:

\begin{gather}
     g_{\alpha,\beta,s}(\theta) = \frac{e^{\frac{1}{2}s^2\theta^2}}{e^{\beta}-e^{\alpha}}G_{\alpha,\beta,s}(\theta), \notag \\
     G_{\alpha,\beta,s}(\theta)=  e^{\alpha}-e^{\beta} + \frac{1}{\theta-1}\left(e^{-\frac{1}{2}s^2\theta^2 +(\beta- \mathbb{E}[Y_{\trunc}])\theta +\alpha} -e^{-\frac{1}{2}s^2\theta^2 +(\alpha- \mathbb{E}[Y_{\trunc}])\theta +\beta}\right), \notag\\
   s_{\text{inf}}(\alpha,\beta)\coloneqq \sqrt{1- \frac{(\beta-\alpha)^2e^{\alpha+\beta}}{\left(e^{\beta}-e^{\alpha}\right)^2}}\,,\,\, s_1(\alpha,\beta)\coloneqq \sqrt{\frac{(\beta-\alpha)^2(\beta-\mathbb{E}[Y_{\trunc}]+2) - \delta_{\alpha,\beta}}{(\beta-\alpha)^2+12}}, \notag\\
   \delta_{\alpha,\beta} = 2(\beta-\alpha)\left(-3\mathbb{E}[Y_{\trunc}]^2+(6\beta-(\beta-\alpha)^2)\mathbb{E}[Y_{\trunc}] +\beta^3-2\alpha\beta^2+\alpha^2\beta+\alpha^2-2\alpha\beta-2\beta^2\right)^{\frac{1}{2}}.\label{exponentialdefinitions}
\end{gather}

Furthermore, by uniqueness and observing that:

\begin{equation*}
    \theta_c=2 \,\,,\,\, \|Y_{\trunc}\|_{\vp}^2=\frac{{ (\beta-\alpha+2)e^{\alpha}+(\beta-\alpha-2)e^{\beta}}}{{2}(e^{\beta}-e^{\alpha})}
\end{equation*}

is a solution of \eqref{Eqexpotosatisfy} we conclude the computation of $\|Y_{\trunc}\|_{\vp}^2$ . Eventually, the statement that $Y_\trunc$ is never strictly sub-Gaussian follows from a direct analysis of the difference between $\|Y_{\trunc}\|_{\vp}^2$ and $\Var[Y_{\trunc}]$ that is detailed in Appendix~\ref{NotStrictlySubGaussian}.

\begin{proof}[Proof of Theorem \ref{thm:main_trunc_expo_standard}] 
We will prove that $\|Y_{\trunc}\|_{\vp}^2$ is the optimal variance proxy if and only if it is the solution to (\ref{Eqexpotosatisfy}) and that such solution is unique. First note that $g_{\alpha,\beta,s}$ and $G_{\alpha,\beta,s}$ have the same sign and that they are smooth functions on~$\mathbb{R}$\footnote{Although the function $g_{\alpha,\beta,s}$ defined in Equation~\eqref{eq:g-definition} itself is undefined at $\theta=1$, setting it equal to $-e^{s^2/2}$ makes it smooth. In the proof of Theorem~\ref{thm:main_trunc_expo_standard}, we may have to avoid $\theta=1$ at some places. Still, it is obvious by continuity of $g_{\alpha,\beta,s}$ that the inequality of Definition~\ref{def:sub-Gaussian} extends at $\theta=1$ if it is proved valid in a neighborhood of this point.}. 
Clearly $s^2$ is a variance proxy if $G_{\alpha,\beta,s}(\theta) \leq 0$ for all $\theta \in \mathbb{R}$. In particular, we have:

\begin{gather}
    g_{\alpha,\beta,s}(0)=0 \,\,,\,\, g_{\alpha,\beta,s}'(0)=0 \,\,,\,\, \forall s>0, \notag\\
    \label{IdG} G_{\alpha,\beta,s}(0)=G'_{\alpha,\beta,s}(0)=0\,\,,\,\, \underset{\theta\to \pm\infty}{\lim} G_{\alpha,\beta,s}(\theta)=
    e^{\alpha}-e^{\beta}<0
\end{gather} 

Thus, another necessary condition for $s^2$ to be a variance proxy is that $G''_{\alpha,\beta,s}(0)\leq 0$, i.e., we get the claimed lower bound for the optimal variance proxy:

\beq\label{sinf} \|Y_{\trunc}\|_{\vp}\geq s_{\text{inf}}(\alpha,\beta)\coloneqq \sqrt{1- \frac{(\beta-\alpha)^2e^{\alpha+\beta}}{\left(e^{\beta}-e^{\alpha}\right)^2}}.\eeq

In the following, we shall thus assume that $s\geq s_{\text{inf}}(\alpha,\beta)$. A simple computation yields that $G'_{\alpha,\beta,s}$ is given by

\begin{equation}\label{Connectiongh}
    G'_{\alpha,\beta,s}(\theta)= \frac{1}{(\theta-1)^2}e^{\alpha \theta+\beta}e^{-\frac{1}{2}s^2\theta^2-\mathbb{E}[Y_{\trunc}]\theta}h_{\alpha,\beta,s}(\theta),
\end{equation}
where:
\begin{gather*}
    h_{\alpha,\beta,s}(\theta)\coloneqq  \left(-s^2\theta^2+(s^2+\beta-\mathbb{E}[Y_{\trunc}])\theta+\mathbb{E}[Y_{\trunc}]-\beta-1\right)e^{(\theta-1)(\beta-\alpha)}\\\notag
    +s^2\theta^2-(s^2+\alpha- \mathbb{E}[Y_{\trunc}])\theta -\mathbb{E}[Y_{\trunc}]+\alpha+1 \notag \\
  h'''_{\alpha,\beta,s}(\theta)=\frac{e^{-(\theta-1)(\beta-\alpha)}}{(\beta-\alpha)}P_{\alpha,\beta,s}(\theta),
\end{gather*}  
with:
\begin{gather*}
    P_{\alpha,\beta,s}(\theta)\coloneqq A_{\alpha,\beta,s}\theta^2+ B_{\alpha,\beta,s} \theta +C_{\alpha,\beta,s},\notag \\
A_{\alpha,\beta,s}\coloneqq -s^2(\beta-\alpha)^2, \notag \\
B_{\alpha,\beta,s}\coloneqq (\alpha-\beta)(\alpha-\beta+6)s^2+(\beta-\alpha)^2(\beta-\mathbb{E}[Y_{\trunc}]),\notag \\
C_{\alpha,\beta,s}\coloneqq 3(\beta-\alpha-2)s^2+(\beta-\alpha)\left((\beta-\alpha-3)\mathbb{E}[Y_{\trunc}]+\alpha\beta-\beta^2+\alpha+2\beta \right).
\end{gather*} 
The discriminant of the second-degree polynomial $P_{\alpha,\beta,s}$ is crucial in establishing the previously mentioned characterization of the optimal variance proxy, as shown in Lemma \ref{lem:char1}. It is given by:
\begin{align*}
    \Delta_{\alpha,\beta,s}\coloneqq &B_{\alpha,\beta,s}^2-4A_{\alpha,\beta,s}C_{\alpha,\beta,s}\\
=&(\beta-\alpha)^2\left((\beta-\alpha)^2+12)s^4-2(\beta-\alpha)^2(\beta-\mathbb{E}[Y_{\trunc}]+2)s^2+(\beta-\alpha)^2(\beta-\mathbb{E}[Y_{\trunc}])^2\right).
\end{align*}
\begin{lemma}\label{lem:char1}
If $\Delta_{\alpha,\beta,s}\leq 0$, then $s^2$ is always a variance proxy because $G_{\alpha,\beta,s}$ has a unique maximum at $\theta=0$ which is vanishing.
On the contrary, if $\Delta_{\alpha,\beta,s}>0$ then $G_{\alpha,\beta,s}$ has exactly two local maximums. One at $\theta=0$ which is vanishing and also another one denoted by $\theta_c(\alpha,\beta,s)$ which is nonzero. 
\end{lemma}

\begin{proof}[Proof of Lemma \ref{lem:char1}] 
For the proof we will need the following immediate results:

\begin{gather}
\underset{\theta\to -\infty}{\lim} h''_{\alpha,\beta,s}(\theta)=2s^2>0, \quad 
\underset{\theta\to +\infty}{\lim} h''_{\alpha,\beta,s}(\theta)=-\infty,\notag\\
\underset{\theta\to \pm\infty}{\lim} h'_{\alpha,\beta,s}(\theta)=-\infty, \quad h'_{\alpha,\beta,s}(1)=0,\notag\\
\underset{\theta\to -\infty}{\lim} h_{\alpha,\beta,s}(\theta)=+\infty, \quad
h_{\alpha,\beta,s}(1)=0,\notag\\
\underset{\theta\to +\infty}{\lim} h_{\alpha,\beta,s}(\theta)=-\infty. \label{ObservationExpoLimit}
\end{gather}

\smallskip

\noindent\textbf{Case $\Delta_{\alpha,\beta,s}\leq 0$.} Let us assume that $\Delta_{\alpha,\beta,s}\leq 0$. Since $A_{\alpha,\beta,s}<0$, we get that $h'''_{\alpha,\beta,s}$ is a strictly negative function on $\mathbb{R}$ except at one point where it vanishes when $\Delta_{\alpha,\beta,s}=0$. This implies that $h''_{\alpha,\beta,s}$ is a strictly decreasing function on $\mathbb{R}$. From the first two equations of \eqref{ObservationExpoLimit}, we conclude that there exists a unique value $\theta_1\in \mathbb{R}$ such that $h''_{\alpha,\beta,s}(\theta_1)=0$. Moreover, $h''_{\alpha,\beta,s}$ is strictly positive on $(-\infty,\theta_1)$ and strictly negative on $(\theta_1,+\infty)$. Thus, $h'_{\alpha,\beta,s}$ is a strictly increasing function on $(-\infty,\theta_1)$ and a strictly decreasing function on $(\theta_1,+\infty)$. Since we know that $h'_{\alpha,\beta,s}(1)=0$ and $\underset{\theta\to \pm\infty}{\lim} h'_{\alpha,\beta,s}(\theta)=-\infty$, we conclude that $h'_{\alpha,\beta,s}$ may have at most two zeros. We have two subcases:
\begin{itemize}
 \item $h'_{\alpha,\beta,s}$ has a double zero at $\theta_1=1$. In this case, from the limit at infinity given by \eqref{ObservationExpoLimit}, $h'_{\alpha,\beta,s}$ is strictly negative on $\mathbb{R}\setminus\{1\}$ and thus $h_{\alpha,\beta,s}$ is a strictly decreasing function on $\mathbb{R}$. From \eqref{ObservationExpoLimit}, we have $h_{\alpha,\beta,s}(1)=0$ so that it implies that $h_{\alpha,\beta,s}$ is strictly positive on $(-\infty,1)$ and strictly negative on $(1,+\infty)$ and so is $G'_{\alpha,\beta,s}$ from \eqref{Connectiongh}. But this is a contradiction to the fact that $G'_{\alpha,\beta,s}(0)=0$ in \eqref{IdG} so that this subcase may be discarded.  
\item  $h'_{\alpha,\beta,s}$ has two distinct zeros, i.e. there exists a unique $\theta_2\neq 1$ such that $h'_{\alpha,\beta,s}(\theta_2)=0$. Let us denote $\td{\theta}_1=\min(1,\theta_1)$ and $\td{\theta}_2=\max(1,\theta_2)$. Then $h'_{\alpha,\beta,s}$ is strictly negative on $(-\infty,\td{\theta}_1)\cup(\td{\theta}_2,+\infty)$ and strictly positive on $(\td{\theta}_1,\td{\theta}_2)$. Thus, $h_{\alpha,\beta,s}$ is strictly decreasing on $(-\infty,\td{\theta}_1)$, strictly increasing on $(\td{\theta}_1,\td{\theta}_2)$ and strictly decreasing on $(\td{\theta}_2,+\infty)$. But from \eqref{ObservationExpoLimit} we have $h_{\alpha,\beta,s}(1)=0$ so that $h_{\alpha,\beta,s}$ has a local extremum at $\theta=1$ that is null. Hence $h_{\alpha,\beta,s}$ has constant sign locally around $\theta=1$. Thus, there exists a unique value $\theta_3\neq 1$ such that $h_{\alpha,\beta,s}(\theta_3)=0$. Moreover, $h_{\alpha,\beta,s}$ is strictly positive on $(-\infty,\theta_3)$ and strictly negative on $(\theta_3,+\infty)$ and so is $G'_{\alpha,\beta,s}$ from \eqref{Connectiongh}. However, from \eqref{IdG} and the fact that $s\geq s_{\text{inf}}(\alpha,\beta)$, we know that $\theta=0$ is a local maximum of $G_{\alpha,\beta,s}$ so that we necessarily have $\theta_3=0$. Hence $G_{\alpha,\beta,s}$ achieves a unique maximum at $\theta=0$ and this maximum is null from \eqref{IdG} so that we conclude that $G_{\alpha,\beta,s}$ is negative on $\mathbb{R}$, i.e. that $s^2$ is a variance proxy.
\end{itemize}
\smallskip

\noindent\textbf{Case $\Delta_{\alpha,\beta,s}> 0$.} Let us assume that $\Delta_{\alpha,\beta,s}> 0$. We get that there exists two distinct values $\theta_0\neq \theta_1$ such that $h'''_{\alpha,\beta,s}(\theta_0)=h'''_{\alpha,\beta,s}(\theta_1)=0$. Moreover, since $A_{\alpha,\beta,s}<0$, we get that $h'''_{\alpha,\beta,s}$ is strictly negative on $(-\infty,\theta_1)\cup(\theta_2,+\infty)$ and strictly positive on $(\theta_1,\theta_2)$. Thus, $h''_{\alpha,\beta,s}$ is strictly decreasing on $(-\infty,\theta_1)$, strictly increasing on $(\theta_1,\theta_2)$ and strictly decreasing on $(\theta_2,+\infty)$. Since $\underset{\theta\to -\infty}{\lim} h''_{\alpha,\beta,s}(\theta)>0$ and $\underset{\theta\to +\infty}{\lim} h''_{\alpha,\beta,s}(\theta)=-\infty$ we conclude that $h''_{\alpha,\beta,s}$ may have one or two (one simple and one double zero) or three distinct zeros depending on the sign of $h'{\alpha,\beta,s}(\theta_1)$ and $h'{\alpha,\beta,s}(\theta_2)$. 

The case when $h''_{\alpha,\beta,s}$ has only one or two zeros is identical to the previous case (because in both cases $h''_{\alpha,\beta,s}$ changes sign only at the unique simple zero) and we refer to its conclusion. Let us thus assume that $h''_{\alpha,\beta,s}$ has three distinct zeros $\hat{\theta}_1<\theta_0<\hat{\theta}_2<\theta_1<\hat{\theta}_3$. Then $h''_{\alpha,\beta,s}$ is strictly positive on $(-\infty,\hat{\theta}_1)\cup(\hat{\theta}_2,\hat{\theta}_3)$ and is strictly negative on  $(\hat{\theta}_1,\hat{\theta}_2 )\cup(\hat{\theta}_3,+\infty)$. Thus, since $\underset{\theta\to \pm\infty}{\lim} h'_{\alpha,\beta,s}(\theta)=-\infty$, $h'_{\alpha,\beta,s}$ may have two distinct zeros or three distinct zeros (in this case, one double zero and two simple zeros) or four distinct zeros. 

Cases corresponding to two or three distinct zeros are identical to the previous case (note that in both cases, $h'_{\alpha,\beta,s}$ only changes sign twice) and we refer to its conclusion. Let us thus assume that $h'_{\alpha,\beta,s}$ has four distinct zeros. These zeros are necessarily simple zeros and from Rolle's Theorem, we conclude that $h_{\alpha,\beta,s}$ may have at most $5$ distinct zeros. However, we have $h_{\alpha,\beta,s}(1)=h'_{\alpha,\beta,s}(1)=0$ so that $\theta=1$ is always a double zero of $h_{\alpha,\beta,s}$. Hence, $h_{\alpha,\beta,s}$ may only change sign at most thrice. However since   $\underset{\theta\to -\infty}{\lim} h_{\alpha,\beta,s}(\theta)=\infty$ and $\underset{\theta\to +\infty}{\lim} h_{\alpha,\beta,s}(\theta)=-\infty$, it cannot change sign twice so that $h_{\alpha,\beta,s}$ may only change sign at one or three distinct locations and so is $G'_{\alpha,\beta,s}$ from \eqref{Connectiongh}. 

The case when $h_{\alpha,\beta,s}$ changes sign once is identical to the previous one and we refer it its conclusion. Let us thus assume that $G'_{\alpha,\beta,s}$ changes sign at three distinct locations denoted $\td{\theta_1}<\td{\theta}_2<\td{\theta}_3$. From the limit at infinity, $G'_{\alpha,\beta,s}$ is necessarily strictly positive on $(-\infty,\td{\theta_1})\cup(\td{\theta_2},\td{\theta_3})$ and strictly negative on $(\td{\theta_1},\td{\theta_2})\cup(\td{\theta_3},+\infty)$. This implies that $G_{\alpha,\beta,s}$ is strictly increasing on $(-\infty,\td{\theta_1})$ and then strictly decreasing on $(\td{\theta}_1,\td{\theta}_2)$, then strictly increasing on $(\td{\theta}_2,\td{\theta}_3)$ and finally strictly decreasing $(\td{\theta}_3,+\infty)$. Hence, $G_{\alpha,\beta,s}$ has two local maxima and one local minimum. However, from \eqref{IdG}, we know that $\theta=0$ is a local extremum of $G_{\alpha,\beta,s}$ and from \eqref{sinf} we have assumed $s\geq s_{\text{inf}}(\alpha,\beta)$ so that $G''_{\alpha,\beta,s}(0)<0$. Hence $\theta=0$ is necessarily a local maximum of $G_{\alpha,\beta,s}$ and its value is null from \eqref{IdG}. Hence we conclude that $G_{\alpha,\beta,s}$ has a vanishing local maximum at $\theta=0$, a local minimum with a strictly negative value, and a second local maximum at $\theta_3\neq 0$. The sign of this second local maximum determines the sign of $G_{\alpha,\beta,s}$ and hence is equivalent to deciding if $s^2$ is a variance proxy. This concludes the proof of Lemma \ref{lem:char1}.
\end{proof}

By the condition on $\Delta_{\alpha,\beta,s}$, we also get the claimed upper bound of the optimal variance proxy from this lemma. Indeed, $\Delta_{\alpha,\beta,s}\leq 0$ is equivalent to have $s\in [s_1(\alpha,\beta),s_2(\alpha,\beta)]$ while $\Delta_{\alpha,\beta,s}>0$ is equivalent to have $s\in (0,s_1(\alpha,\beta))\cup (s_2(\alpha,\beta),+\infty)$. Therefore $\|Y_{\trunc}\|_{\vp}\leq s_1(\alpha,\beta)$, with:

\begin{align*}
s_1(\alpha,\beta) \coloneqq \sqrt{\frac{(\beta-\alpha)^2(\beta-\mathbb{E}[Y_{\trunc}]+2) - \delta_{\alpha,\beta}}{(\beta-\alpha)^2+12}}, \quad s_2(\alpha,\beta) \coloneqq \sqrt{\frac{(\beta-\alpha)^2(\beta-\mathbb{E}[Y_{\trunc}]+2) +\delta_{\alpha,\beta}}{(\beta-\alpha)^2+12}},
\end{align*}

\normalsize{and} where $\delta_{\alpha,\beta}$ was defined in ($\ref{exponentialdefinitions}$). It is obvious that the position of the second maximum $\theta_c(\alpha,\beta,s)$ depends smoothly on $s$. Moreover, we know that the optimal variance proxy exists and is non-zero because of \eqref{sinf}. The previous analysis implies that for $s<\|Y_{\trunc}\|_{\vp}$ we must have $\Delta_{\alpha,\beta,s}>0$ and the existence of $\theta_c(\alpha,\beta,s) \neq 0$ with $G_{\alpha,\beta,s}(\theta_c(\alpha,\beta,s))>0$. Thus, taking the limit $s\to \|Y_{\trunc}\|_{\vp}$ with $s<\|Y_{\trunc}\|_{\vp}$ implies that  $G_{\alpha,\beta,\|Y_{\trunc}\|_vp}(\theta_c(\alpha,\beta,\|Y_{\trunc}\|_vp))=0$. Notice that $\theta_c(\alpha,\beta,\|Y_{\trunc}\|_vp)=\pm\infty$ is not possible because we have $\underset{\theta\to \pm \infty}{\lim}G_{\alpha,\beta,s}(\theta)=e^{\alpha}-e^{\beta}<0$ independent of $s$. Hence we must have

\begin{equation*}
    G_{\alpha,\beta,\|Y_{\trunc}\|_{\vp}}(\theta_c(\alpha,\beta,\|Y_{\trunc}\|_{\vp}))=0.
\end{equation*}

To finish it remains to prove that $\theta_c(\alpha,\beta,\|Y_{\trunc}\|_{\vp}) \neq 0$. Indeed, let us assume by contradiction that $\theta_c(\alpha,\beta,\|Y_{\trunc}\|_{\vp})=0$. The previous analysis shows that for $s<\|Y_{\trunc}\|_{\vp}$, $G'_{\alpha,\beta,s}$ has two distinct zeros inside $(0,\theta_c(\alpha,\beta,s))$ and thus by Rolle's theorem that $G''_{\alpha,\beta,s}$ has at least one zero in $(0,\theta_c(\alpha,\beta,s))$. Thus at the limit, we would get $G''_{\alpha,\beta,\|Y_{\trunc}\|_{\vp}}(0)=0$, i.e. $\theta=0$ would be at least a triple zero of $G_{\alpha,\beta,\|Y_{\trunc}\|_{\vp}}$. In fact, in order to remain locally negative around $\theta=0$, we would necessarily have a zero of order four, i.e. $G'''_{\alpha,\beta,\|Y_{\trunc}\|_{\vp}}(0)=0$. Hence $\theta=0$ would be a zero of order four of $g_{\alpha,\beta,\|Y_{\trunc}\|_{\vp}}$. This would imply $\|Y_{\trunc}\|_{\vp}=s_{\text{inf}}(\alpha,\beta)$ (triple zero using the definition of $s_{\text{inf}}(\alpha,\beta)$ and $g'''_{\alpha,\beta,s_{\text{inf}}(\alpha,\beta)}(0)$ (zero of order four). However we have:
\begin{equation*}
    g'''_{\alpha,\alpha+\epsilon,s_{\text{inf}}(\alpha,\beta)}(0)=\frac{2e^{3\epsilon}-(\epsilon^3+6)e^{2\epsilon} +(6-\epsilon^3)e^{\epsilon}-2}{(e^{\epsilon}-1)^3}=\frac{P(\epsilon)}{(e^{\epsilon}-1)^3}.
\end{equation*}

From Lemma \ref{Lemmaepsilon} in Appendix \ref{sec:app:expo} we have that  $g'''_{\alpha,\alpha+\epsilon,s_{\text{inf}}(\alpha,\alpha+\epsilon)}(0)>0$ for all $\epsilon>0$ so that $g'''_{\alpha,\beta,s_{\text{inf}}(\alpha,\beta)}(0)>0$ for any $\alpha<\beta$ i.e., we cannot have $\theta_c(\alpha,\beta,\|Y_{\trunc}\|_{\vp})=0$. We conclude that the optimal variance proxy is characterized by the unique solution of the below system of equations, which is coherent with the illustration of Figure~\ref{fig:Expo}. This concludes the proof of Theorem \ref{thm:main_trunc_expo_standard}.
\end{proof}

\begin{figure}
    \centering
    \includegraphics[width=.6\textwidth]{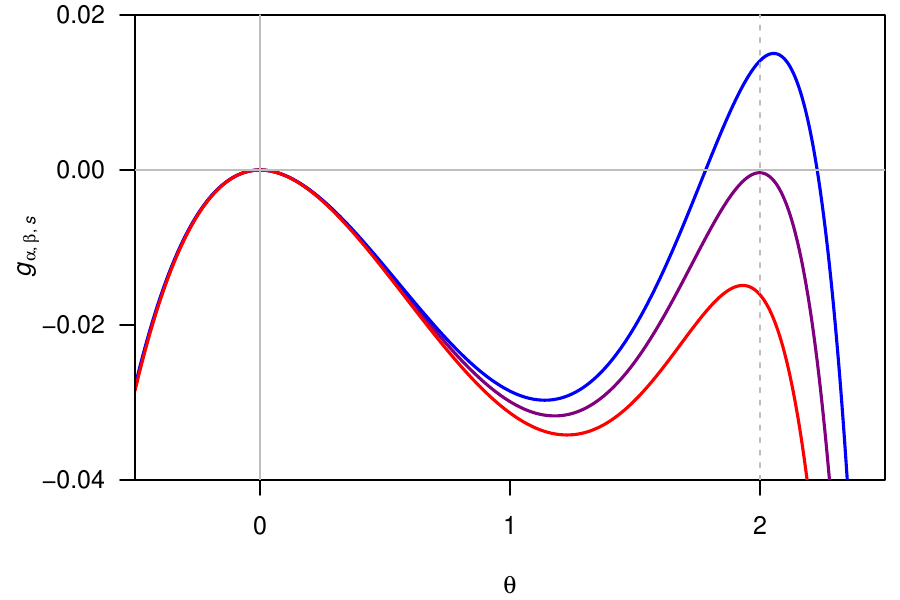} 
    \caption{\textbf{Illustration of the proof of Theorem~\ref{thm:main_trunc_expo_standard}.}  
    Function $g_{\alpha,\beta,s}$ defined in Equation~\eqref{eq:g-definition} is represented with $\alpha=1$ and $\beta=4$ and for different values of $s$. 
    Non-positivity of the  $g_{\alpha,\beta,s}$ function is equivalent to $s^2$ being a variance proxy. 
    In purple is the optimal variance proxy case $s^2= s_c^2 = \Vert Y_\trunc\Vert_\vp^2$ ($s= s_c\approx0.8107$). Red represents $s=0.812>s_c$ (with $s^2$ being a variance proxy) and blue represents $s=0.8095<s_c$  (with $s^2$ not being a variance proxy).}
    \label{fig:Expo}
\end{figure}

\section{Future research directions}\label{sec:directions}

In summary, our work has identified the optimal sub-Gaussian variance proxy for truncated Gaussian and truncated exponential random variables, while also delineating the conditions under which strict sub-Gaussianity may or may not be observed in these truncated distributions.

Moving forward, there are several avenues for extending this work. Firstly, exploring additional commonly encountered distributions beyond Gaussian and exponential would broaden the scope of our findings. Additionally, investigating the truncation of multivariate distributions, such as the multivariate Gaussian distribution, could provide valuable insights into more complex scenarios.

Furthermore, considering the sub-Weibull property of truncated random variables represents a promising direction for future research  \citep{vladimirova2020subweibull}. This concept, which serves as a generalization of the sub-Gaussian property, offers an intriguing framework for further understanding the behavior of truncated distributions across various contexts.

\clearpage

\section*{Acknowledgments}
\support
\newpage
\appendix

\renewcommand{\theequation}{\thesection.\arabic{equation}}

\section{Proofs for truncated Gaussian random variables}\label{sec:app:gauss}

This entire Appendix~\ref{sec:app:gauss} is devoted to proving Lemma~\ref{lem:h_strict_concave}, showing that the function $h_{\alpha,\beta}$ is strictly concave on $[\theta_0,+\infty)$ for any finite $\alpha<\beta$. 

\begin{proof}[Proof of Lemma \ref{lem:h_strict_concave}]
To begin with, observe that it is sufficient to check the concavity of the function only for the case when $\alpha = -\beta$ (i.e. $\theta_0:=\frac{\alpha+\beta}{2}=0$). Indeed, suppose that $h''_{-\beta,\beta}(\theta)<0$ for all $\beta \in \mathbb{R}_{>0}$ and $\theta > 0$. Then, this implies that $h_{-\frac{\beta-\alpha}{2},\frac{\beta-\alpha}{2}}(\theta) = h_{\alpha,\beta}(\theta + \theta_0)$ is strictly concave for $\theta \geq 0$, i.e. that $x\mapsto h_{\alpha,\beta}(x)$ is strictly concave for $x \geq \theta_0$. 
\subsection{Some notations and preliminary results}
Let us recall that we have:
\begin{equation}
    \label{RefId}  
\begin{split}
&\Phi(x)\coloneqq \int_{-\infty}^x \frac{1}{\sqrt{2\pi}}e^{-\frac{1}{2}s^2}ds,\\
&F_{-\beta,\beta}(\theta)\coloneqq \Phi(\beta-\theta)-\Phi(-\beta-\theta),\\
&f_{-\beta,\beta}(\theta)\coloneqq \ln\left(\frac{F_{-\beta,\beta}(\theta)}{2\Phi(\beta)-1}\right),\\
&F'_{-\beta,\beta}(\theta)=\phi(\theta+\beta)-\phi(\theta-\beta)=\frac{1}{\sqrt{2\pi}}\left(e^{-\frac{1}{2}(\theta+\beta)^2}-e^{-\frac{1}{2}(\theta-\beta)^2}\right),\\
&F''_{-\beta,\beta}(\theta)=\frac{1}{\sqrt{2\pi}}\left(-(\beta+\theta)e^{-\frac{1}{2}(\theta+\beta)^2}+(\theta-\beta)e^{-\frac{1}{2}(\theta-\beta)^2}\right),\\
&F'''_{-\beta,\beta}(\theta)=\frac{1}{\sqrt{2\pi}}\left(e^{-\frac{1}{2}(\theta+\beta)^2}\left[(\beta+\theta)^2-1\right] + e^{-\frac{1}{2}(\theta-\beta)^2}\left[1-(\beta-\theta)^2\right]\right),\\
&h_{-\beta,\beta}(\theta)\coloneqq f'_{-\beta,\beta}(\theta)=\frac{F'_{-\beta,\beta}(\theta)}{F_{-\beta,\beta}(\theta)},\\
&h'_{-\beta,\beta}(\theta)=\frac{F_{-\beta,\beta}(\theta)F''_{-\beta,\beta}(\theta) -(F'_{-\beta,\beta}(\theta))^2 }{F_{-\beta,\beta}(\theta)^2},\\
&h''_{-\beta,\beta}(\theta)=\frac{(F_{-\beta,\beta}(\theta))^2 F'''_{-\beta,\beta}(\theta)-3F_{-\beta,\beta}(\theta)F'_{-\beta,\beta}(\theta)F''_{-\beta,\beta}(\theta)+2(F'_{-\beta,\beta}(\theta))^3 }{F_{-\beta,\beta}(\theta)^3},\\
&\underset{\theta\to +\infty}{\lim} F'_{-\beta,\beta}(\theta)=\underset{\theta\to +\infty}{\lim} F''_{-\beta,\beta}(\theta)=\underset{\theta\to +\infty}{\lim} F'''_{-\beta,\beta}(\theta)=0,\\
&\underset{{\tiny{\theta\to +\infty}}}{\lim} F_{-\beta,\beta}(\theta)=0.
\end{split}
\end{equation}
It is obvious by definition that $F_{-\beta,\beta}$ is strictly positive on $\mathbb{R}$. It is also straightforward to obtain the important linear relation
\beq \label{Fthird} F'''_{-\beta,\beta}(\theta)=-P_2'(\theta) F''_{-\beta,\beta}(\theta)-P_2(\theta)F'_{-\beta,\beta}(\theta), \eeq
where
\beq P_2(\theta)\coloneqq \theta^2+1-\beta^2. \eeq
In order to prove the strict concavity of $h_{-\beta,\beta}$ on $\mathbb{R}_{\geq 0}$, we shall need information on the function $F_{-\beta,\beta}$ and its derivatives. Thus, we shall first prove the following lemma.

\begin{lemma}\label{LemmaFthird}We have the following results:
\begin{itemize}
\item We have $F'_{-\beta,\beta}(\theta)<0$ for any $\beta>0$ and any $\theta\in \mathbb{R}_{>0}$.
\item For $0<\beta\leq\sqrt{3}$: The function $F'''_{-\beta,\beta}$ has only one zero on $\mathbb{R}_{>0}$ that we shall denote $\theta_1$. Moreover, it is strictly positive on $(0,\theta_1)$ and strictly negative on $(\theta_1,+\infty)$. Consequently, the function $F''_{-\beta,\beta}$ is strictly increasing on $[0,\theta_1]$ from $F''_{-\beta,\beta}(0)=-\frac{2\beta}{\sqrt{2\pi}}e^{-\frac{1}{2}\beta^2}<0$ to $F''_{-\beta,\beta}(\theta_1)>0$ and strictly decreasing on $(\theta_1,+\infty)$ from $F''_{-\beta,\beta}(\theta_1)>0$ to $0$. In particular $F''_{-\beta,\beta}$ has only one zero on $\mathbb{R}_{>0}$ denoted $\alpha_1$ satisfying $\alpha_1>\theta_1$.
\item  For $\beta>\sqrt{3}$: The function $F'''_{-\beta,\beta}$ has two distinct zeros on $\mathbb{R}_{>0}$ denoted $\theta_1<\theta_2$ on $\mathbb{R}_{>0}$. Moreover, it is strictly negative on $(0,\theta_1)\cup(\theta_2,+\infty)$ and strictly positive on $(\theta_1,\theta_2)$. Consequently, the function $F''_{-\beta,\beta}$ is decreasing on $(0,\theta_1)$ from $F''_{-\beta,\beta}(0)=-\frac{2\beta}{\sqrt{2\pi}}e^{-\frac{1}{2}\beta^2}<0$ to $F''_{-\beta,\beta}(\theta_1)<F''_{-\beta,\beta}(0)<0$. $F''_{-\beta,\beta}$ is then increasing on $(\theta_1,\theta_2)$ up to $F''_{-\beta,\beta}(\theta_2)$. It finally decreases from $F''_{-\beta,\beta}(\theta_2)$ to its limit $0$ at $\theta\to +\infty$. In particular, we must have $F''_{-\beta,\beta}(\theta_2)>0$ and $F''_{-\beta,\beta}$ has only one zero on $\mathbb{R}_{>0}$ denoted $\alpha_1$ that verifies $\alpha_2\in (\theta_1,\theta_2)$.
\end{itemize} 
\end{lemma}

\begin{proof}[Proof of Lemma \ref{LemmaFthird}]
The first point is obvious from \eqref{RefId}. Let us observe that 
\beq F''_{-\beta,\beta}(\theta)= \frac{2}{\sqrt{2\pi}}\exp\left(-\frac{1}{2}(\beta^2+\theta^2)\right)\left(\theta \sinh(\beta\theta)-\beta \cosh(\beta\theta)\right).\eeq
In particular $F''_{-\beta,\beta}(0)=-\frac{2\beta}{\sqrt{2\pi}}\exp\left(-\frac{1}{2}\beta^2\right)<0$ and $\underset{\theta\to +\infty}{\lim}F''_{-\beta,\beta}(\theta)=0$. Let us assume by contradiction that $F'''_\beta$ does not vanish on $\mathbb{R}_{>0}$. This implies that $F_{-\beta,\beta}''$ would be a strictly monotonous function and therefore would be a strictly increasing function because of the values of $F''_{-\beta,\beta}(0)$ and $F''_{-\beta,\beta}(+\infty)$. Since $\underset{\theta\to +\infty}{\lim}F''_{-\beta,\beta}(\theta)=0$, it implies that $F''_{-\beta,\beta}$ would be a strictly negative function on $\mathbb{R}_{>0}$ so that $F'_{-\beta,\beta}$ would be a strictly decreasing function on $\mathbb{R}_{>0}$. This is a contradiction with the fact that $\underset{\theta\to +\infty}{\lim}F'_{-\beta,\beta}(\theta)=0$ and $F'_{-\beta,\beta}$ is strictly negative on $\mathbb{R}_{>0}$.  Therefore, we conclude that there is at least one value of $\theta_c\in \mathbb{R}_{>0}$ for which $F'''_{-\beta,\beta}(\theta_c)=0$. This point shall be used later to determine the sign of some quantities by contradiction\footnote{In fact, the reasoning implies that there exists a non-trivial interval on which $F'''_{-\beta,\beta}$ is strictly positive.}. Let us now compute

\beq F'''_{-\beta,\beta}\left(\theta=\frac{s}{\beta}\right)=\frac{1}{\sqrt{2\pi}}\exp\left(-\frac{1}{2}\left(\beta^2+\frac{s^2}{\beta^2}\right)\right)e^{-s}\left(\left(1-\left(\frac{s}{\beta}-\beta\right)^2\right)e^{2s}  -\left(1-\left(\frac{s}{\beta}+\beta\right)^2\right) \right).\eeq

Therefore, the sign and zeros of $F'''_{-\beta,\beta}$ are equivalent to the sign and zeros of $g_\beta(s)$ defined by
\beq g_\beta(s)\coloneqq \left(1-\left(\frac{s}{\beta}-\beta\right)^2\right)e^{2s}  -1+\left(\frac{s}{\beta}+\beta\right)^2.\eeq

We have  
\begin{align*}
    g_\beta'(s)&=-\frac{2(\beta^4-2(s+1) \beta^2+s^2+s)}{\beta^2}e^{2s}+\frac{2(\beta^2+s)}{\beta^2},\\
g''_\beta(s)&=-\frac{2(2\beta^4-4\beta^2s-6\beta^2+2s^2+4s+1)}{\beta^2}e^{2s}+\frac{2}{\beta^2} ,\\
g'''_\beta(s)&=-\frac{4(2s^2-(4\beta^2-6)s+2\beta^4-8\beta^2+3)}{\beta^2}e^{2s}.
\end{align*}

\sloppy{The discriminant of the numerator of $g'''_\beta$ is given by $\Delta=4(4\beta^2+3)>0$. Hence $g'''_\beta$ has two distinct zeros on $\mathbb{R}$ and we need to study the sign of these zeros. Let us denote
\beq P(\beta)\coloneqq 2\beta^4-8\beta^2+3 \,\, ,\,\, S(\beta)\coloneqq 4\beta^2-6,\eeq
so that the product (resp. the sum) of the two zeros of $g'''_\beta$ are given by $\frac{1}{2}P(\beta)$ (resp. $\frac{1}{2}S(\beta)$). $P(\beta)$ is strictly positive on $\left(0,\frac{\sqrt{8-2\sqrt{10}}}{2}\right)\cup\left(\frac{\sqrt{8+2\sqrt{10}}}{2},+\infty\right)$ and strictly negative on $\left(\frac{\sqrt{8-2\sqrt{10}}}{2},\frac{\sqrt{8+2\sqrt{10}}}{2} \right)$.  $S(\beta)$ is strictly positive on $\left(\frac{\sqrt{6}}{2},+\infty\right)$ and strictly negative on $\left(0,\frac{\sqrt{6}}{2}\right)$. We shall also observe:}
\begin{align*}
&\lim_{s\to +\infty}g_\beta(s)=\lim_{s\to +\infty}g'_\beta(s)=\lim_{s\to +\infty}g''_\beta(s)=\lim_{s\to +\infty}g'''_\beta(s)=-\infty,\\
&g'''_\beta(0)=-\frac{4P(\beta)}{\beta^2},\\
&g''_\beta(0)=4(3-\beta^2),\\
&g'_\beta(0)=2(3-\beta^2) ,\\
&g_\beta(0)=0.
\end{align*}
Let us now study the following four cases.\smallskip

\noindent\textbf{Case $\beta>\frac{\sqrt{8+2\sqrt{10}}}{2}$.} Here, $g'''_\beta$ has two roots on $\mathbb{R}$ and their product in strictly positive and their sum is strictly positive. Thus, $g'''_\beta$ has two distinct roots on $\mathbb{R}_{>0}$ denoted $0<s_1<s_2$. It is strictly negative on $(0,s_1)\cup(s_2,+\infty)$ and strictly positive in $(s_1,s_2)$. Therefore $g''_\beta$ strictly decreases on $(0,s_1)$ and since $g''_\beta(0)<0$ it remains strictly negative on $(0,s_1)$. $g''_\beta$ is then strictly increasing from $g''_\beta(s_1)$ up to $g''_\beta(s_2)$. We have $g''_\beta(s_2)=\frac{2}{\beta^2}\left(e^{2\beta^2+\sqrt{4\beta^2+3}-3}(\sqrt{4\beta^2+3}-1) +1\right)>0$. Since $g''_\beta$ diverges towards $-\infty$ when $s\to +\infty$, we end up with fact that there exist exactly two values $(s_3,s_4)\in (\mathbb{R}_{>0})^2$ such that $g''_\beta(s_3)=g''_\beta(s_4)=0$ and they satisfy $s_3<s_2<s_4$. Moreover $g''_\beta$ is strictly positive on $\left(s_3,s_4\right)$ and strictly negative on $\left(0,s_3\right)\cup\left(s_4,+\infty\right)$. Therefore $g'_\beta$ strictly decreases on $\left(0,s_3\right)$ and since $g'_\beta(0)<0$ it remains strictly negative on $(0,s_3)$. Note that we necessarily have $g'_\beta(s_4)>0$. Indeed, if $g'_{\beta}(s_4)\leq 0$ then $g'_\beta$ would be strictly negative on $\mathbb{R}_{>0}$ so that $g_\beta$ would be strictly decreasing and since $g_\beta(0)=0$ it would not vanish on $\mathbb{R}_{>0}$ which is a contradiction (because we have proved that $F'''_{-\beta,\beta}$ must vanish at least once). Since $g'_\beta(s_4)>0$, we obtain that there exists exactly two values $(s_5,s_6)\in (\mathbb{R}_{>0})^2$ such that $g'_{\beta}(s_5)=g'_\beta(s_6)=0$ and they satisfy $s_3<s_5<s_4<s_6$. $g_\beta$ is thus strictly decreasing on $(0,s_5)$ and since $g_{\beta}(0)=0$ it remains strictly negative. Again $g_\beta(s_6)$ must be strictly positive otherwise $g_{\beta}$ would not vanish and nor would $F'''_{-\beta,\beta}$ which is a contradiction. Therefore, the variations of $g_\beta$ implies that there exists exactly two values $(s_7,s_8)\in (\mathbb{R}_{>0})^2$ such that $g_\beta(s_7)=g_\beta(s_8)=0$ and they satisfy $s_5<s_7<s_6<s_8$.
\smallskip

\noindent\textbf{Case $\beta\in \Big(\sqrt{3},\frac{\sqrt{8+2\sqrt{10}}}{2}\,\Big]$.} Here, $g'''_\beta$ has two roots on $\mathbb{R}$ and their product in negative or null and their sum is strictly positive. This means that $g'''_\beta$ has exactly one root on $\mathbb{R}_{>0}$ that we shall denote $s_2>0$. $g'''_\beta$ is strictly positive on $(0,s_2)$ and strictly negative on $(s_2,+\infty)$. Since $g''_\beta(0)<0$ and $g''_\beta(s_2)=\frac{2}{\beta^2}\left(e^{2\beta^2+\sqrt{4\beta^2+3}-3}(\sqrt{4\beta^2+3}-1) +1\right)>0$ and $\underset{s\to +\infty}{\lim} g''_\beta(s)=-\infty$, we get that there exists exactly two distinct values $(s_3,s_4)\in (\mathbb{R}_{>0})^2$ such that $g''_\beta(s_3)=g''_\beta(s_4)=0$ and they satisfy $s_3<s_2<s_4$. Moreover, $g'_\beta$ is strictly decreasing on $(0,s_3)$, strictly increasing on $(s_3,s_4)$ and strictly decreasing on $(s_4,+\infty)$. We have $g'_\beta(0)<0$ so that $g'_\beta(s_3)<0$. Similarly to the previous case, we must have $g'_\beta(s_4)>0$ otherwise $g_\beta$ would not vanish on $\mathbb{R}_{>0}$ (it would be strictly decreasing on $\mathbb{R}_{>0}$ with $g_\beta(0)=0$). We obtain that there exists exactly two values $(s_5,s_6)\in (\mathbb{R}_{>0})^2$ such that $g'_\beta(s_5)=g'_\beta(s_6)=0$ and they satisfy $s_3<s_5<s_4<s_6$. $g_\beta$ is then strictly decreasing from $g_\beta(0)=0$ to $g_\beta(s_5)<0$. We must have $g_\beta(s_6)>0$ otherwise $g_\beta$ would not vanish on $\mathbb{R}_{>0}$ (which is a contradiction since we know that $F'''_{-\beta,\beta}$ has at least one zero on $\mathbb{R}_{>0}$). Therefore, the variations of $g_\beta$ implies that there exists exactly two values $(s_7,s_8)\in (\mathbb{R}_{>0})^2$ such that $g_\beta(s_7)=g_\beta(s_8)=0$ and they satisfy $s_5<s_7<s_6<s_8$.
\smallskip

\noindent\textbf{Case $\beta\in \Big(\frac{\sqrt{8-2\sqrt{10}}}{2},\sqrt{3}\Big]$.} Here, $g'''_\beta$ has two roots on $\mathbb{R}$ and their product is strictly negative. This means that $g'''_\beta$ has exactly one root on $\mathbb{R}_{>0}$ that we shall denote $s_2>0$. $g'''_\beta$ is strictly positive on $(0,s_2)$ and strictly negative on $(s_2,+\infty)$. We have $g''_\beta(0)\geq 0$ and $\underset{s\to +\infty}{\lim} g''_\beta(s)=-\infty$. Thus, since $g_\beta''$ is strictly increasing and then strictly decreasing, we get that there exists exactly one value $s_4\in \mathbb{R}_{>0}$ such that $g''_\beta(s_4)=0$ and it satisfies $s_2<s_4$.  Moreover, $g'_\beta$ is strictly increasing on $(0,s_4)$ with $g_\beta'(0)\geq 0$. It then strictly decreases from $g'_\beta(s_4)>0$ towards $-\infty$. Hence, there exists exactly one value $s_6>s_4$ such that $g'_\beta(s_6)=0$ and $g'_\beta$ is strictly positive on $(0,s_6)$ and strictly negative on $(s_6,+\infty)$. Thus, $g_\beta$ is strictly increasing from $g_\beta(0)=0$ to $g_\beta(s_6)>0$ and then strictly decreasing towards $-\infty$. We conclude that there exists only one value $s_8$ on $\mathbb{R}_{>0}$ such that $g_\beta(s_8)=0$. 
\smallskip

\noindent\textbf{Case $\beta\in \Big(0,\frac{\sqrt{8-2\sqrt{10}}}{2}\,\Big]$.} Here, $g'''_\beta$ has two roots on $\mathbb{R}$, and their product is positive or null, while their sum is strictly negative. This implies that $g'''_\beta$ does not vanish on $\mathbb{R}_{>0}$ and thus is strictly negative. $g''_\beta$ is thus strictly decreasing on $\mathbb{R}_{>0}$ with $g''_\beta(0)>0$ and $g_\beta''(+\infty)=-\infty$. We get that there exists exactly one value $s_4\in \mathbb{R}_{>0}$ such that $g''_\beta(s_4)=0$. Moreover, $g'_\beta$ is strictly increasing on $(0,s_4)$ with $g_\beta'(0)>0$. It then strictly decreases from $g'_\beta(s_4)>0$ towards $-\infty$.  Hence, there exists exactly one value $s_6>s_4$ such that $g'_\beta(s_6)=0$ and $g'_\beta$ is strictly positive on $(0,s_6)$ and strictly negative on $(s_6,+\infty)$. Thus, $g_\beta$ is strictly increasing from $g_\beta(0)=0$ to $g_\beta(s_6)>0$ and then strictly decreasing towards $-\infty$. We conclude that there exists only one value $s_8$ on $\mathbb{R}_{>0}$ such that $g_\beta(s_8)=0$. 
\smallskip

Summarizing the four different cases and using $\theta=\frac{s}{\beta}$, we conclude that for $\beta\in\Big(0,\sqrt{3}\Big]$, $F'''_{-\beta,\beta}$ has exactly one root $\theta_1$ on $\mathbb{R}_{>0}$ and it is strictly positive on $\left(0,\theta_1\right)$ and it is strictly negative on $\left(\theta_1,+\infty\right)$. On the contrary, for $\beta>\sqrt{3}$, $F'''_{-\beta,\beta}$ has exactly two distinct roots on $\mathbb{R}_{>0}$ denoted $\theta_1<\theta_2$ and it is strictly positive on $\left(\theta_1,\theta_2\right)$ and strictly negative on $\left(0,\theta_1\right)\cup\left(\theta_2,+\infty\right)$. The rest of the lemma is then obvious, which concludes the proof of Lemma \ref{LemmaFthird}.
\end{proof}

\subsection{A sufficient condition for strict concavity}
Let us first prove that for any $\beta>0$, the function $h_{-\beta,\beta}$ is strictly concave in a positive neighborhood of $\theta=0$. It is a straightforward computation by taking Taylor series around $\theta=0$ to observe that
\beq h''_{-\beta,\beta}(\theta)\overset{\theta\to 0}{=}-\frac{\beta\exp\left(-\frac{\beta^2}{2}\right)}{(2\Phi(\beta)-1)^2}\left(\frac{6\beta}{\pi}\exp\left(-\frac{\beta^2}{2}\right)+\frac{2(\beta^2-3)}{\sqrt{2\pi}}(2\Phi(\beta)-1) \right)\theta +O(\theta^2).
\eeq
The leading order is strictly negative for $\beta\geq \sqrt{3}$ because both terms are negative or null. For $\beta\in [0,\sqrt{3})$ we define
\beq f_1(\beta)\coloneqq \frac{6\beta}{\pi(\beta^2-3)}\exp\left(-\frac{\beta^2}{2}\right)+\frac{2}{\sqrt{2\pi}}(2\Phi(\beta)-1),\eeq
whose derivative is $f_1'(\beta)=-\frac{4\beta^4 \exp\left(-\frac{\beta^2}{2}\right)}{\pi(\beta^2-3)^2}<0$. Therefore $f_1$ is strictly decreasing on $[0,\sqrt{3})$ and since $f_1(0)=0$ we get that $f_1$ is strictly negative on $[0,\sqrt{3})$. This implies that the leading order of $h''_{-\beta,\beta}(\theta)$ as $\theta\to 0$ is strictly negative on $\mathbb{R}_{>0}$. In particular, we get that for any $\beta>0$, there exists a positive neighborhood of $\theta=0$ on which $h_{-\beta,\beta}$ is strictly concave.

\medskip

Let us now reformulate the problem of strict concavity more simply. From \eqref{RefId}, the equation $h_{-\beta,\beta}''(\theta)=0$ is equivalent to
\beq\label{ZeroConcavity}Z_{\beta}(\theta)\coloneqq (F_{-\beta,\beta}(\theta))^2 F'''_{-\beta,\beta}(\theta)-3F_{-\beta,\beta}(\theta)F'_{-\beta,\beta}(\theta)F''_{-\beta,\beta}(\theta)+2(F'_{-\beta,\beta}(\theta))^3=0.\eeq
We shall denote 
\begin{align*}
     a_\beta(\theta)&\coloneqq F'''_{-\beta,\beta}(\theta),\\
b_{\beta}(\theta)&\coloneqq -3F'_{-\beta,\beta}(\theta)F''_{-\beta,\beta}(\theta),\\
c_{\beta}(\theta)&\coloneqq 2(F'_{-\beta,\beta}(\theta))^3,
\end{align*}
so that
\beq \label{ZeroConcavity2}Z_{\beta}(\theta)= a_\beta(\theta)(F_{-\beta,\beta}(\theta))^2+b_\beta(\theta)F_{-\beta,\beta}(\theta)+c_\beta(\theta).\eeq
Let us observe that if we can prove that $Z_{\beta}(\theta)$ does not vanish on $\mathbb{R}_{>0}$, then it proves that $h_{-\beta,\beta}$ is strictly concave on $\mathbb{R}_{>0}$. Indeed, the function $Z_\beta(\theta)$ is continuous in $\theta$. Moreover, we have proved that it is strictly negative in a positive neighborhood of $\theta=0$ so that if it does not vanish on $\mathbb{R}_{>0}$, then the intermediate value theorem implies that it must remain strictly negative on $\mathbb{R}_{>0}$. Therefore, we obtain the following sufficient condition to prove strict concavity of $h_{-\beta,\beta}$ on $\mathbb{R}_{>0}$.

\begin{proposition}\label{SuffCond1}Let $\beta>0$. Proving that $Z_\beta$ does not vanish on $\mathbb{R}_{>0}$ is a sufficient condition to proving the strict concavity of $h_{-\beta,\beta}$ on $\mathbb{R}_{>0}$.
\end{proposition}

The next step is to use the fact that the r.h.s. of \eqref{ZeroConcavity2} may be seen as a polynomial of degree $2$ in $F_{-\beta,\beta}(\theta)$. In particular, zeros  $\theta_c$ of $Z_\beta$ are either zeros of $F'''_\beta$ or solutions of the system
\begin{equation}
    \label{EqToSatisfy}  
\begin{split}
&9(F'_{-\beta,\beta}(\theta_c))^2(F''_{-\beta,\beta}(\theta_c))^2-8(F'_{-\beta,\beta}(\theta_c))^3F'''_{-\beta,\beta}(\theta_c)\geq 0,\\
&F_{-\beta,\beta}(\theta_c)=\frac{3}{2}\frac{F'_{-\beta,\beta}(\theta_c)F''_{-\beta,\beta}(\theta_c)}{F'''_{-\beta,\beta}(\theta_c)}\\
&\quad\pm\frac{1}{2F'''_{-\beta,\beta}(\theta_c)}\sqrt{ 9(F'_{-\beta,\beta}(\theta_c))^2(F''_{-\beta,\beta}(\theta_c))^2-8(F'_{-\beta,\beta}(\theta_c))^3F'''_{-\beta,\beta}(\theta_c)}\\
&=-\frac{b_\beta(\theta_c)}{2a_\beta(\theta_c)} \pm \frac{1}{2a_\beta(\theta_c)}\sqrt{b_\beta(\theta_c)^2-4a_\beta(\theta_c)c_\beta(\theta_c)}.
\end{split}
\end{equation}
The first inequality in \eqref{EqToSatisfy} is necessary otherwise \eqref{ZeroConcavity} which is polynomial of degree $2$ in $F_{-\beta,\beta}(\theta_c)$ would have no real roots and thus \eqref{ZeroConcavity} would not have solutions on $\mathbb{R}_{>0}$ ending the proof. Therefore, we shall define
\begin{align*}
    T_{\beta,\pm}(\theta)&\coloneqq  -F_{-\beta,\beta}(\theta)+\frac{3}{2}\frac{F'_{-\beta,\beta}(\theta)F''_{-\beta,\beta}(\theta)}{F'''_{-\beta,\beta}(\theta)} \\
    &\quad\pm\frac{1}{2F'''_{-\beta,\beta}(\theta)}\sqrt{ 9(F'_{-\beta,\beta}(\theta))^2(F''_{-\beta,\beta}(\theta))^2-8(F'_{-\beta,\beta}(\theta))^3F'''_{-\beta,\beta}(\theta)}\\
&=-\frac{b_\beta(\theta)}{2a_\beta(\theta)} \pm \frac{1}{2a_\beta(\theta)}\sqrt{b_\beta(\theta)^2-4a_\beta(\theta)c_\beta(\theta)}-F_{-\beta,\beta}(\theta),
\end{align*}
and we have the following proposition.

\begin{proposition}\label{PropSufficientCondition} Let $\beta>0$. A sufficient condition to prove strict concavity of $h_{-\beta,\beta}$ on $\mathbb{R}_{>0}$ is to prove that the function $S_\beta$ is strictly positive on $\mathbb{R}_{>0}$ where
\begin{equation}
    \label{EqToSatisfy2}
\begin{split}
S_\beta(\theta)&\coloneqq  9\theta (F_{-\beta,\beta}''(\theta))^5\\
&\quad+\left(42\theta^2-9\right)F_{-\beta,\beta}'(\theta) (F_{-\beta,\beta}''(\theta))^4\\
&\quad-15\left(\beta^2-\frac{79}{15}\theta^2+2\right)\theta (F_{-\beta,\beta}'(\theta))^2(F_{-\beta,\beta}''(\theta))^3\\
&\quad+\left(75\theta^4-(42\beta^2+36)\theta^2-\beta^4+12\beta^2-3\right)(F_{-\beta,\beta}'(\theta))^3(F_{-\beta,\beta}''(\theta))^2 \\
&\quad+4\left(9\theta^4-\left(10\beta^2+\frac{9}{2}\right)\theta^2+\beta^4+\frac{9}{2}\beta^2-\frac{3}{2}\right)\theta(F_{-\beta,\beta}'(\theta))^4F_{-\beta,\beta}''(\theta)\\
&\quad+\left(7\theta^6-(13\beta^2+3)\theta^4+(5\beta^4+6\beta^2-3)\theta^2+(\beta^2-1)^3\right)(F_{-\beta,\beta}'(\theta))^5.
\end{split}
\end{equation}
\end{proposition}

\begin{proof}[Proof of Proposition \ref{PropSufficientCondition}]
As explained in Proposition \ref{SuffCond1}, a sufficient condition to prove the strict concavity of $h_{-\beta,\beta}$ on $\mathbb{R}_{>0}$ is to show that $Z_\beta$ does not vanish on $\mathbb{R}_{>0}$. Moreover, the previous discussion implies that zeros of $Z_\beta$ are either zeros of $F'''_{-\beta,\beta}$ or zeros of $T_{\pm,\beta}$. Let us first prove that the functions $T_{\beta,+}$ and $T_{\beta,-}$ do not vanish on $\mathbb{R}_{>0}\setminus \mathcal{R}_\beta$ where $\mathcal{R}_\beta$ is the set of zeros of $F'''_{-\beta,\beta}$. Let us observe that for any $\theta\in \mathbb{R}_{>0}\setminus \mathcal{R}_\beta$:
\begin{equation}
     \label{TprimeEq}
\begin{split}
T_{\beta,\pm}'(\theta)&= 	
    \left(
    \frac{3F'_{-\beta,\beta}(\theta)F''_{-\beta,\beta}(\theta)
    \pm\sqrt{ 9(F'_{-\beta,\beta}(\theta))^2(F''_{-\beta,\beta}(\theta))^2-8(F'_{-\beta,\beta}(\theta))^3F'''_{-\beta,\beta}(\theta)}}
    {2F'''_{-\beta,\beta}(\theta)} 
    \right)' -F'_{-\beta,\beta}(\theta)\\
&= \frac{1}{2\sqrt{b_\beta(\theta)^2-4a_\beta(\theta)c_\beta(\theta)}}
\Big(\left(a_\beta'(\theta)b_\beta(\theta)-a_{\beta}(\theta)b'_\beta(\theta)-2F''_{-\beta,\beta}(\theta)\right)\sqrt{b_\beta(\theta)^2-4a_\beta(\theta)c_\beta(\theta)}\\
&\qquad\qquad\mp\left(b_\beta(\theta)^2a_\beta'(\theta)-b_\beta(\theta)a_\beta(\theta)b_{\beta}'(\theta)+2a_\beta(\theta)c'_\beta(\theta)-2a_\beta(\theta)c_\beta(\theta)a'_\beta(\theta)\right)\Big)
\end{split}
\end{equation}
\normalsize{is} only expressed in terms of $F_{-\beta,\beta}'$ and its derivatives that are classical functions. Zeros of $T'_{\beta,\pm}$ must satisfy (taking the square of the numerator of \eqref{TprimeEq} to remove the $\pm$ sign):
\begin{multline*}
     \left(b_\beta(\theta)^2a_\beta'(\theta)-b_\beta(\theta)a_\beta(\theta)b_{\beta}'(\theta)+2a_\beta(\theta)c'_\beta(\theta)-2a_\beta(\theta)c_\beta(\theta)a'_\beta(\theta)\right)^2\\
-\left(a_\beta'(\theta)b_\beta(\theta)-a_{\beta}(\theta)b'_\beta(\theta)-2F''_{-\beta,\beta}(\theta)\right)^2\left(b_\beta(\theta)^2-4a_\beta(\theta)c_\beta(\theta)\right)=0.
\end{multline*}
Replacing $F_{-\beta,\beta}'''$ and $F_{-\beta,\beta}^{(4)}$ in terms of $F_{-\beta,\beta}'$ and $F_{-\beta,\beta}''$ using \eqref{Fthird} gives after a tedious computation that the r.h.s. is of the form $8F'''_{-\beta,\beta}(\theta)^2F'_{-\beta,\beta}(\theta)S_\beta(\theta)$. 
Therefore zeros of $T'_{\beta,\pm}$ are among those of $S_\beta$. Simple asymptotic expansions around $\theta=0$ and $\theta\to +\infty$ provide the following results:

\begin{equation}\label{Identities}
\begin{split}
&\lim_{\theta\to +\infty}T_{\beta,\pm}(\theta)=0,\\
&T_{\beta,+}(0)=-\frac{\left(\sqrt{2\pi}\left(2\phi(\beta) -1\right) (\beta^2-3) \exp\left( \frac{\beta^2}{2} \right)+6\beta\right) \exp\left(-\frac{\beta^2}{2} \right) \sqrt{2}}{2\sqrt{\pi}(\beta^2-3)} \,\, ,\forall \, \beta\in\mathbb{R}_{>0}\setminus\{\sqrt{3}\},\\
&T'_{\beta,+}(\theta)=-\frac{2\beta}{3\sqrt{2\pi}}e^{-\frac{1}{2}\beta^2}\theta+ O(\theta^2) \,\, ,\,\,\forall \, \beta\in\mathbb{R}_{>0}\setminus\{\sqrt{3}\},\\
&T_{\beta,-}(0)=-2\Phi(\beta) +1+O(\theta^2) \,\, ,\,\,\forall \, \beta\in\mathbb{R}_{>0}\setminus\{\sqrt{3}\}, \\
&T'_{\beta,-}(\theta)=-\frac{4\beta^5}{3\sqrt{2\pi}(\beta^2-3)^2}e^{-\frac{1}{2}\beta^2}\theta+ O(\theta^2) \,\, ,\,\,\forall \, \beta\in\mathbb{R}_{>0}\setminus\{\sqrt{3}\},\\
&T_{\sqrt{3},+}(\theta)=\frac{3e^{-\frac{3}{2}}\sqrt{6}}{\sqrt{\pi}\theta^2} +O(1), \\
&T'_{\sqrt{3},+}(\theta)=-\frac{6e^{-\frac{3}{2}}\sqrt{6}}{\sqrt{\pi}\theta^3} + O(\theta),\\
&T_{\sqrt{3},-}(\theta)=-2\Phi(\sqrt{3})+1 +O(\theta^2),\\
&T'_{\sqrt{3},-}(\theta)=-\frac{1}{3\sqrt{\pi}}e^{-\frac{3}{2}} \sqrt{6}\theta +O(\theta^2).
\end{split}
\end{equation}

\normalsize{In} particular, in all cases, we get that the functions $\left(T'_{\beta,-},T'_{\beta,+}\right)$ are always strictly negative in a positive neighborhood of $\theta=0$ for any value of $\beta>0$. 
Let us show that proving that $S_\beta$ is strictly positive on $\mathbb{R}_{>0}$ is a sufficient condition to get that both functions $(T_{\beta,-},T_{\beta,+})$ do not vanish on $\mathbb{R}_{>0}$. Indeed, if we assume that $S_\beta(\theta)$ is strictly positive on $\mathbb{R}_{>0}$ then it implies that $T'_{\beta,\pm}$ cannot vanish. Depending on the value of $\beta$, we have three cases:
\begin{itemize}\item For $\beta<\sqrt{3}$: we have that $T'_{\beta,\pm}$ is strictly negative on $(0,\theta_1)$ so that $T_{\beta,\pm}$ is a strictly decreasing function on $(0,\theta_1)$. Note that $T_{\beta,-}$ diverges at $\theta_1$ and changes sign (because $3F'_{-\beta,\beta}(\theta_1)F''_{-\beta,\beta}(\theta_1)<0$ from Lemma \ref{LemmaFthird}). Moreover, from \eqref{Identities}, we have $T_{\beta,-}(0)<0$. This implies that $\underset{\theta\to \theta_{1,-}}{\lim} T_{\beta,-}=-\infty$ and  $\underset{\theta\to \theta_{1,+}}{\lim} T_{\beta,-}=+\infty$. Therefore, the sign of $T'_{\beta,-}$ is necessarily negative on $(\theta_1,+\infty)$ so that $T_{\beta,-}$ is strictly decreasing on $(\theta_1,+\infty)$. Since its limit is zero at infinity, it remains strictly positive on $(\theta_1,+\infty)$. Thus, we conclude that $T_{\beta,-}$ never vanishes on $\mathbb{R}_{>0}$. The situation for $T_{\beta,+}$ is simpler. 
Indeed, $T_{\beta,+}$ is a smooth function at  $\theta=\theta_1$ (because $3F'_{-\beta,\beta}(\theta_1)F''_{-\beta,\beta}(\theta_1)<0$ from Lemma \ref{LemmaFthird}). Therefore the sign of $T_{\beta,+}'$ remains constant on $\mathbb{R}_{>0}$ and thus $T_{\beta,+}$ is a decreasing function on $\mathbb{R}_{>0}$. Since its limit at infinity is null, we get that it remains strictly positive on $\mathbb{R}_{>0}$. In both cases, $T_{\beta,\pm}$ does not vanish on $\mathbb{R}_{>0}$. 
\item For $\beta=\sqrt{3}$: we have that $T'_{\sqrt{3},\pm}$ is strictly negative in $(0,\theta_1)$ so that $T_{\sqrt{3},\pm}$ is a strictly decreasing function on $(0,\theta_1)$. Note that $T_{\sqrt{3},-}$ diverges at $\theta_1$ and changes sign (because $3F'_{-\sqrt{3},\sqrt{3}}(\theta_1)F''_{-\sqrt{3},\sqrt{3}}(\theta_1)<0$ from Lemma \ref{LemmaFthird}). Moreover, from \eqref{Identities}, we have $T_{\sqrt{3},-}(0)<0$. This implies that $\underset{\theta\to \theta_{1,-}}{\lim} T_{\sqrt{3},-}=-\infty$ and  $\underset{\theta\to \theta_{1,+}}{\lim} T_{\sqrt{3},-}=+\infty$. Therefore, the sign of $T'_{\sqrt{3},-}$ is necessarily negative on $(\theta_1,+\infty)$ so that $T_{\sqrt{3},-}$ is strictly decreasing on $(\theta_1,+\infty)$. Since its limit is zero at infinity, it remains strictly positive on $(\theta_1,+\infty)$. Thus, we conclude that $T_{\sqrt{3},-}$ never vanishes on $\mathbb{R}_{>0}$. The situation for $T_{\sqrt{3},+}$ is simpler. We have from \eqref{Identities} that $T_{\sqrt{3},+}(0)>0$ and $T_{\sqrt{3},+}$ is a smooth function at  $\theta=\theta_1$ (because $3F'_{-\sqrt{3},\sqrt{3}}(\theta_1)F''_{-\sqrt{3},\sqrt{3}}(\theta_1)<0$ from Lemma \ref{LemmaFthird}). Therefore the sign of $T_{\sqrt{3},+}'$ remains constant on $\mathbb{R}_{>0}$ and $T_{\sqrt{3},+}$ is a decreasing function on $\mathbb{R}_{>0}$. Since its limit is null, we get that it remains strictly positive on $\mathbb{R}_{>0}$. In both cases, $T_{\sqrt{3},\pm}$ does not vanish on $\mathbb{R}_{>0}$.
\item For $\beta>\sqrt{3}$, we get that $T'_{\beta,\pm}$ is strictly negative in $(0,\theta_1)$. Note that $T_{\beta,+}$ is smooth at $\theta_2$ but diverges and changes sign at $\theta=\theta_1$ (because $3F'_{-\beta,\beta}(\theta_1)F''_{-\beta,\beta}(\theta_1)<0$ and $3F'_{-\beta,\beta}(\theta_2)F''_{-\beta,\beta}(\theta_2)>0$ from Lemma \ref{LemmaFthird}). Moreover, we have from \eqref{Identities} $T_{\beta,+}(0)<0$ so that $\underset{\theta\to \theta_{1,-}}{\lim} T_{\beta,+}=-\infty$ and  $\underset{\theta\to \theta_{1,+}}{\lim} T_{\beta,+}=+\infty$. Therefore $T_{\beta,+}$ must be decreasing on $(\theta_1,+\infty)$ and since its limit is $0$ at infinity we conclude that it is strictly positive on $(\theta_1,+\infty)$. The situation for $T_{\beta,-}$ is similar. It is smooth at $\theta_1$ but diverges and changes sign at $\theta_2$ (because $3F'_{-\beta,\beta}(\theta_1)F''_{-\beta,\beta}(\theta_1)<0$ and $3F'_{-\beta,\beta}(\theta_2)F''_{-\beta,\beta}(\theta_2)>0$ from Lemma \ref{LemmaFthird}). We have $T_{\beta,-}(0)<0$ so that it is strictly negative on $(0,\theta_2)$ and $\underset{\theta\to \theta_{2,-}}{\lim} T_{\beta,-}=-\infty$ and  $\underset{\theta\to \theta_{2,+}}{\lim} T_{\beta,-}=+\infty$. Therefore it must strictly decrease on $(\theta_2,+\infty)$ and since its limit is $0$ at infinity, we end up with the fact that it is strictly positive on $(\theta_2,+\infty)$. In both cases, $T_{\beta,\pm}$ does not vanish on $\mathbb{R}_{>0}$.  
\end{itemize}

Thus, we conclude that proving that $S_\beta$ is strictly positive on $\mathbb{R}_{>0}$ is a sufficient condition to get that $T_{\beta,\pm}$ does not vanish on $\mathbb{R}_{>0} \setminus \mathcal{R}_\beta$. Let us prove that the fact that $S_\beta$ is strictly positive on $\mathbb{R}_{>0}$ also excludes the zeros of $F'''_{-\beta,\beta}$ as potential zeros of $Z_\beta$. Under the assumption that $S_\beta$ is strictly positive on $\mathbb{R}_{>0}$ we have
\begin{itemize}\item For $\beta>\sqrt{3}$, $F'''_{-\beta,\beta}$ has exactly two zeros denoted $\theta_1<\theta_2$ on $\mathbb{R}_{>0}$. As explained above, $T_{\beta,+}$ is a smooth function at $\theta_2$ and we have $T_{\beta,+}(\theta_2)=-\frac{c_\beta(\theta_2)}{b_{\beta}(\theta_2)}- F_{-\beta,\beta}(\theta_2)$. We have proved above that $T_{\beta,+}$ is strictly positive on $(\theta_1,+\infty)$ so that $T_{\beta,+}(\theta_2)= -\frac{c_\beta(\theta_2)}{b_{\beta}(\theta_2)}- F_{-\beta,\beta}(\theta_2)>0$. However, if $\theta_2$ was a zero of $Z_\beta$, then we would have from \eqref{ZeroConcavity} $F_{-\beta,\beta}(\theta_2)+\frac{c_\beta(\theta_2)}{b_{\beta}(\theta_2)}=0$ leading to a contradiction. Therefore $\theta_2$ is not a zero of $Z_\beta$. Similarly, we have proved that $T_{\beta,-}$ is a smooth function at $\theta_1$ and we have $T_{\beta,-}(\theta_1)=-\frac{c_\beta(\theta_1)}{b_{\beta}(\theta_1)}- F_{-\beta,\beta}(\theta_1)$. Moreover, we have proved above that $T_{\beta,-}$ is strictly negative on $(0,\theta_2)$ so that $T_{\beta,-}(\theta_1)= -\frac{c_\beta(\theta_1)}{b_{\beta}(\theta_1)}- F_{-\beta,\beta}(\theta_1)<0$. However, if $\theta_1$ was a zero of $Z_\beta$, then we would have from \eqref{ZeroConcavity} $F_{-\beta,\beta}(\theta_1)+\frac{c_\beta(\theta_1)}{b_{\beta}(\theta_1)}=0$ leading to a contradiction. Therefore $\theta_1$ is not a zero of $Z_\beta$. 
\item For $\beta\leq \sqrt{3}$: $F'''_{-\beta,\beta}$ has exactly one zero denoted $\theta_1$ on $\mathbb{R}_{>0}$. As explained above, $T_{\beta,+}$ is a smooth function at  $\theta=\theta_1$ and it is strictly positive on $\mathbb{R}_{>0}$. In particular we have $T_{\beta,+}(\theta_1)=-\frac{c_\beta(\theta_1)}{b_{\beta}(\theta_1)}- F_{-\beta,\beta}(\theta_1)>0$. However, if $\theta_1$ was a zero of $Z_\beta$, then we would have from \eqref{ZeroConcavity} $F_{-\beta,\beta}(\theta_1)+\frac{c_\beta(\theta_1)}{b_{\beta}(\theta_1)}=0$ leading to a contradiction. Therefore $\theta_1$ is not a zero of $Z_\beta$. 
\end{itemize}
This concludes the proof of Proposition \ref{PropSufficientCondition}.
\end{proof}

\subsection{Proof of the sufficient condition}

From Proposition \ref{PropSufficientCondition}, a sufficient condition to have strict concavity of $h_{-\beta,\beta}$ on $\mathbb{R}_{>0}$ is to prove that $S_\beta$ is strictly positive on $\mathbb{R}_{>0}$. In this section, we shall propose a sufficient condition to obtain this result. Let us first observe that:
\begin{align*}
     F'(\theta)
&= -\frac{2}{\sqrt{2\pi}}e^{-\frac{1}{2}\theta^2-\frac{1}{2}\beta^2}\sinh(\beta\theta),\\
F''(\theta)
&=\frac{2}{\sqrt{2\pi}}e^{-\frac{1}{2}\theta^2-\frac{1}{2}\beta^2}\left(\theta \sinh(\beta\theta)-\beta \cosh(\beta\theta)\right).
\end{align*}
The term $\frac{2}{\sqrt{2\pi}}e^{-\frac{1}{2}\theta^2-\frac{1}{2}\beta^2}$ factors out of $S_\beta$ because we have homogeneous powers. Therefore, let us thus rewrite $S_\beta(\theta)=2\left(\frac{2}{\sqrt{2\pi}}e^{-\frac{1}{2}\theta^2-\frac{1}{2}\beta^2}\right)^5\td{S}_\beta(\theta)$. We have:
\begin{align*}
     \td{S}_\beta(\theta)&= \sinh(5\beta\theta)+ \left(4\beta^6+4(3\theta^2+6)\beta^4+12\beta^2-5\right)\sinh(3\beta\theta)\\
&+ \left(-12\beta^6+4(3\theta^2+18)\beta^4-36\beta^2+10)\right)\sinh(\beta\theta)\\
&-4\left(3\beta^2+\theta^2+6\right)\beta^3\theta\cosh(3\beta\theta)\\
&+4\left(-33\beta^2+\theta^2+6\right)\beta^3\theta\cosh(\beta\theta).
\end{align*}
Proving that $S_{\beta}$ is strictly positive on $\mathbb{R}_{>0}$ is equivalent to prove that $\td{S}_\beta$ is strictly positive on $\mathbb{R}_{>0}$. Let us perform the following change of variables: $s=\beta \theta$ and define $A_\beta(s)\coloneqq \td{S}_\beta\left(\frac{s}{\beta}\right)$. Since $\beta>0$, it is obvious that proving that $\td{S}_\beta$ is strictly positive on $\mathbb{R}_{>0}$ is equivalent to proving that $A_{\beta}$ is also strictly positive on $\mathbb{R}_{>0}$. We obtain:
\begin{align*}
    A_\beta(s)&=\sinh(5s)+\left(4\beta^6+24\beta^4+12s^2\beta^2+12\beta^2-5\right)\sinh(3s)\\
&+ \left(-12\beta^6+72\beta^4+12s^2\beta^2-36\beta^2+10\right)\sinh(s)\\
&-4(3\beta^4+6\beta^2+s^2)s\cosh(3s)+4(-33\beta^4+6\beta^2+s^2)s\cosh(s).
\end{align*}
Let us observe that $B_s(\mu)\coloneqq A_{\beta}(s)$ is a polynomial of degree $3$ in $\mu\coloneqq \beta^2$:
\begin{align*}
    B_s(\mu)&\coloneqq A_{\mu^2}(s)\\
&=4\left(\sinh(3s)-3\sinh(s)\right)\mu^3 \\
&\quad+4\left(6\sinh(3s)+18\sinh(s)-33s\cosh(s)-3s\cosh(3s)\right)\mu^2\\
&\quad+4\left(3(s^2+1)\sinh(3s)+3(s^2-3)\sinh(s)+6s\cosh(s)-6s\cosh(3s)\right)\mu\\
&\quad+\sinh(5s)-5\sinh(3s)+10\sinh(s)+4s^3\cosh(s)-4s^3\cosh(3s).
\end{align*}
Let us first notice that $B_s(0)=\sinh(5s)-5\sinh(3s)+10\sinh(s)+4s^3\cosh(s)-4s^3\cosh(3s)>0$ for any $s>0$ so that $B_s$ is strictly positive in a positive neighborhood of $s=0$. Indeed, we have the following lemma.
\begin{lemma}\label{LemmaBs0}For any $s>0$, we have 
\beq \sinh(5s)-5\sinh(3s)+10\sinh(s)+4s^3\cosh(s)-4s^3\cosh(3s)>0.\eeq
\end{lemma}
\begin{proof}[Proof of Lemma \ref{LemmaBs0}] 
One may rewrite
\begin{align*}
    &\sinh(5s)-5\sinh(3s)+10\sinh(s)+4s^3\cosh(s)-4s^3\cosh(3s)=16\sinh(s)^2\left( \sinh(s)^3-s^3\cosh(s)\right)\\
&=16\sinh(s)^2\cosh(s)\left( \frac{\sinh(s)^3}{\cosh(s)}-s^3\right).
\end{align*}
Let us then define $a(s)\coloneqq \frac{\sinh(s)^3}{\cosh(s)}-s^3$ for $s>0$. We have
\begin{align*}
    a'(s)&= \frac{2\cosh(s)^4+(-3s^2-1)\cosh(s)^2-1}{\cosh(s)^2},\\
a''(s)&=\frac{(4\sinh(s)\cosh(s)^4-6s\cosh(s)^3+2\sinh(s))}{\cosh(s)^3},\\
a^{(3)}(s)&=\frac{2\sinh(s)^4(4\cosh(s)^2+3)}{\cosh(s)^4}>0,
\end{align*}
so that $a''$ is strictly increasing on $\mathbb{R}_{>0}$ and since $a''(0)=0$, we get that $a''$ is strictly positive on $\mathbb{R}_{>0}$. This implies that $a'$ is strictly increasing on $\mathbb{R}_{>0}$ and since $a'(0)=0$ we end up with $a'$ strictly positive on $\mathbb{R}_{>0}$. In the end, $a$ is strictly increasing on $\mathbb{R}_{>0}$ and $a(0)=0$ so that $a$ is strictly positive on $\mathbb{R}_{>0}$, ending the proof of Lemma~\ref{LemmaBs0}.
\end{proof}

Then, we observe that the leading coefficient of the polynomial $\mu\mapsto B_s(\mu)$ is given by $\sinh(3s)-3\sinh(s)$ and is obviously strictly positive for any $s>0$. We want to prove that for any $s>0$, $\mu \mapsto B_s(\mu)$ is a strictly positive function on $\mathbb{R}_{>0}$. We have:
\begin{align*}
    B_s'(\mu)&=12\Big[ \left(\sinh(3s)-3\sinh(s)\right)\mu^2\\
&+2 \left(2\sinh(3s)+6\sinh(s)-11s\cosh(s)-s\cosh(3s)\right)\mu\\
&+(s^2+1) \sinh(3s)+(s^2-3)\sinh(s)+2s \cosh(s)-2s\cosh(3s)\Big].
\end{align*}
The discriminant $144\Delta(s)$ of this polynomial of degree two is given by $\Delta(s)\coloneqq -64\cosh(s)\td{\Delta}(s)$ with
\begin{align*}
    \td{\Delta}(s)&=-3\cosh(s)^5+2s\sinh(s)\cosh(s)^4-6s^2\cosh(s)^3+14s\sinh(s)\cosh(s)^2\\
&-3\cosh(s)^3-3s^2\cosh(s)+2s\sinh(s)+6\cosh(s).
\end{align*}
Let us assume that $\td{\Delta}$ is strictly positive on $\mathbb{R}_{>0}$ so that $\Delta(s)$ is strictly negative, i.e. $B_s'$ does not vanish on $\mathbb{R}$ and thus remains strictly positive on $\mathbb{R}_{>0}$ (because its leading coefficient is strictly positive). Since we have proved in Lemma \ref{LemmaBs0} that $B_s(0)>0$, we obtain that $B_s$ is strictly positive on $\mathbb{R}_{>0}$ which is equivalent to say that for any $(\beta,s)\in \mathbb{R}_{>0}\times \mathbb{R}_{>0}$: $A_\beta(s)$ is strictly positive. Therefore, {a sufficient condition to obtain that $S_\beta$ is strictly positive on $\mathbb{R}_{>0}$ is that $\td{\Delta}$ is strictly positive on $\mathbb{R}_{>0}$}.  In order to prove this sufficient condition, let us observe that 
\begin{align*}
    \td{\Delta}'(s)&=\sinh(s)\Big[ 10s\cosh(s)^3\sinh(s)-13\cosh(s)^4-18s^2\cosh(s)^2\\
&+32s \sinh(s)\cosh(s)+5\cosh(s)^2-3s^2+8\Big]\\
&\coloneqq \sinh(s)Q(s),
\end{align*}
with
\begin{align*}
    Q'(s)&=2\sinh(s)\Big[20s\sinh(s)\cosh(s)^2-21\cosh(s)^3-18s^2\cosh(s)+19s\sinh(s)+21\cosh(s)\Big]\\
&\coloneqq 2\sinh(s)R(s),
\end{align*}
with
\begin{equation*}
    R(s)\coloneqq 20s\sinh(s)\cosh(s)^2-21\cosh(s)^3-18s^2\cosh(s)+19s\sinh(s)+21\cosh(s).
\end{equation*}
Let us assume that $R(s)>0$ for any $s>0$, then $Q$ is a strictly increasing function on $\mathbb{R}_{>0}$ with $Q(0)=0$ so that it is strictly positive on $\mathbb{R}_{>0}$. This implies that $\td{\Delta}$ is a strictly increasing function on $\mathbb{R}_{>0}$ so that since $\td{\Delta}(0)=0$ we get that $\td{\Delta}$ is also a strictly positive function on $\mathbb{R}_{>0}$. Therefore, we have the following result.

\begin{proposition}\label{SSS}Let $\beta>0$. A sufficient condition to obtain that $S_\beta$ is strictly positive on $\mathbb{R}_{>0}$ is to prove that $s\mapsto R(s)\coloneqq 20s\sinh(s)\cosh(s)^2-21\cosh(s)^3-18s^2\cosh(s)+19s\sinh(s)+21\cosh(s)$ is strictly positive on $\mathbb{R}_{>0}$.
\end{proposition}

Finally we may prove this sufficient condition using the following proposition.

\begin{proposition}\label{PropRPos}Let $\beta>0$. The function $s\mapsto R(s)\coloneqq 20s\sinh(s)\cosh(s)^2-21\cosh(s)^3-18s^2\cosh(s)+19s\sinh(s)+21\cosh(s)$ is strictly positive on $\mathbb{R}_{>0}$.
\end{proposition}
\begin{proof}[Proof of Proposition \ref{PropRPos}]
Let us first rewrite:
\begin{align*}
    R(s)&=20s\sinh(s)\cosh(s)^2-21\cosh(s)^3-18s^2\cosh(s)+19s\sinh(s)+21\cosh(s)\\
&=20s\sinh(s)\cosh(s)^2-21\cosh(s)\sinh(s)^2-18s^2\cosh(s)+19s\sinh(s)\\
&=-\left(18\cosh(s)s^2- (20\sinh(s)\cosh(s)^2+19\sinh(s))s+21\cosh(s)\sinh(s)^2\right),
\end{align*}
and observe that
\begin{equation*}
    R(s)=\frac{64}{15}s^6+O(s^7),
\end{equation*}
so that $R$ is strictly positive in a positive neighborhood of zero. Moreover, equation $R(s)=0$ may be seen as a polynomial of degree two in $s$. In other words $R(s)=0$ with $s>0$ is equivalent to
\beq \label{Finals} s=\frac{20\sinh(s)\cosh(s)^2+19\sinh(s)}{36\cosh(s)}\pm\frac{\sinh(s)}{36\cosh(s)}\sqrt{400\cosh(s)^4-752\cosh(s)^2+361}.\eeq
Let us define
\begin{equation*}
    R_\pm(s)\coloneqq s- \frac{20\sinh(s)\cosh(s)^2+19\sinh(s)}{36\cosh(s)}-\pm\frac{\sinh(s)}{36\cosh(s)}\sqrt{400\cosh(s)^4-752\cosh(s)^2+361}.
\end{equation*}
Note in particular that 
\begin{equation}
    \label{Rpm0}
    \begin{split}
         R_+(s)&=\frac{64}{45}s^5+O(s^6),\\
         R_-(s)&=-\frac{1}{6}s-\frac{7}{18}s^3+ O(s^4).
    \end{split}
\end{equation}
Moreover, we have
\footnotesize{
\begin{equation*}
    R_{\pm}'(s)=\frac{(-40\cosh(s)^4+56\cosh(s)^2-19)\sqrt{400\cosh(s)^4-752\cosh(s)^2+361}\pm(800\cosh(s)^6-1152\cosh(s)^4+361)}{36\cosh(s)^2\sqrt{400\cosh(s)^4-752\cosh(s)^2+361}}.
\end{equation*}
}
\normalsize{Therefore, zeros} of $R'_{\pm}$ must satisfy
\footnotesize{
\begin{equation*}
    \left(-40\cosh(s)^4+56\cosh(s)^2-19\right)^2\left(400\cosh(s)^4-752\cosh(s)^2+361\right)-\left(800\cosh(s)^6-1152\cosh(s)^4+361\right)^2=0,
\end{equation*}
}
\normalsize{which} is equivalent to
\begin{equation*}
    -576\sinh(s)^4\cosh(s)^2\left(2000\cosh(s)^4-3781\cosh(s)^2+1805\right)=0.
\end{equation*}
Since the discriminant of $2000X^2-3781X+1805$ is negative (equal to $-144039$), we conclude that $R_{\pm}'$ does not vanish on $\mathbb{R}_{>0}$ and thus has a constant sign on $\mathbb{R}_{>0}$. In particular, from \eqref{Rpm0}, we get that $R'_+$ is strictly positive on $\mathbb{R}_{>0}$ so that $R_+$ is strictly increasing and since $R_+(0)=0$ we conclude that $R_+$ is strictly positive on $\mathbb{R}_{>0}$. Similarly, from \eqref{Rpm0} $R'_-$ is strictly negative on $\mathbb{R}_{>0}$ so that $R_-$ is strictly decreasing and since $R_-(0)=0$ we conclude that $R_-$ is strictly negative on $\mathbb{R}_{>0}$. In both cases, the function $R_\pm$ do not vanish on $\mathbb{R}_{>0}$ so that \eqref{Finals} cannot be satisfied on $\mathbb{R}_{>0}$ and eventually $R$ does not vanish on $\mathbb{R}_{>0}$. Since we have shown that it is strictly positive in a positive neighborhood of zero and since it is a smooth function, we conclude that function $R$ is strictly positive on $\mathbb{R}_{>0}$. This concludes the proof of Proposition \ref{PropRPos}.  
\end{proof}

Finally, we conclude from Proposition~\ref{SSS} and Proposition~\ref{PropRPos} that for any $\beta>0$, $S_\beta$ is a strictly positive function on $\mathbb{R}_{>0}$ so that from Proposition~\ref{PropSufficientCondition} we obtain that $h_{-\beta,\beta}$ is a strictly concave function on $\mathbb{R}_{>0}$. This concludes the proof of Lemma~\ref{lem:h_strict_concave} stated in the main part of the paper.
\end{proof}

\section{Proofs for truncated exponential random variables}\label{sec:app:expo}

\begin{lemma}\label{Lemmaepsilon}For any $\epsilon>0$, we have
\begin{equation*}
    P(\epsilon)\coloneqq 2e^{3\epsilon}-(\epsilon^3+6)e^{2\epsilon} +(6-\epsilon^3)e^{\epsilon}-2>0.
\end{equation*}
\end{lemma}
\begin{proof}[Proof of Lemma \ref{Lemmaepsilon}]
We have:
\begin{align*}
    P'(\epsilon)&=e^{3\epsilon}\left((6-3\epsilon^2-\epsilon^3)e^{-2\epsilon} -(12+3\epsilon^2+2\epsilon^3)e^{-\epsilon}+6\right)\coloneqq e^{3\epsilon}Q(\epsilon),\\
Q'(\epsilon)&=e^{-\epsilon}\left[\left(2\epsilon^3+3\epsilon^2-6\epsilon-12\right)e^{-\epsilon}+2\epsilon^3-3\epsilon^2-6\epsilon+12 \right]\coloneqq e^{-\epsilon}R(\epsilon).
\end{align*}
It is then straightforward to compute $R^{(4)}(\epsilon)=e^{-\epsilon}\epsilon(2\epsilon^2-21\epsilon+42)$ whose roots are $\epsilon_{\pm}=\frac{21}{4}\pm \sqrt{105}{4}$. Thus $R'''$ is strictly increasing on $(0,\epsilon_-)$ and then strictly decreasing on $(\epsilon_-,\epsilon_+)$ and finally increasing on $(\epsilon_+,+\infty)$. Since $R'''(0)=0$ and $R'''(\epsilon_+)>0$ we conclude that $R'''$ is strictly positive on $(0,+\infty)$. Thus, $R''$ is strictly increasing on $\mathbb{R}_{\geq 0}$ and since $R''(0)=0$ it is strictly positive on $\mathbb{R}_{>0}$. Hence $R'$ is strictly increasing on $\mathbb{R}_{\geq 0}$ and since $R'(0)=0$ it is strictly positive on $\mathbb{R}_{>0}$. Eventually, $R$ is strictly increasing on $\mathbb{R}_{\geq 0}$ and since $R(0)=0$, we get that $R$ is strictly positive on $\mathbb{R}_{>0}$ so that $Q$ is strictly increasing on $\mathbb{R}_{\geq 0}$. In the end, since $Q(0)=0$, $Q$ is strictly positive on $\mathbb{R}_{>0}$ so that $P$ is strictly increasing on $\mathbb{R}_{\geq 0}$. Since $P(0)=0$, we conclude that $P$ is strictly positive on $\mathbb{R}_{>0}$, which concludes the proof of Lemma \ref{Lemmaepsilon}.
\end{proof}

\subsection{Proving that the truncated exponential is never strictly sub-Gaussian}\label{NotStrictlySubGaussian}
Let us study the sign of 
\begin{equation*}
    \|Y_{\trunc}\|_{\vp}^2-\text{Var}[Y_{\trunc}]
=\frac{e^{\alpha+\beta}}{2\left(e^{\beta}-e^{\alpha}\right)}\left((\alpha-\beta-4)e^{\alpha-\beta}+(\beta-\alpha-4)e^{\beta-\alpha}+2(\beta-\alpha)^2+8\right).
\end{equation*}
Observe that the last term is only a function of $\epsilon=\beta-\alpha>0$. Thus, the sign of $\|Y_{\trunc}\|_{\vp}^2-\text{Var}[Y_{\trunc}]$ is the same as the sign of the function $K$ on $\mathbb{R}_{>0}$ defined by
\begin{equation*}
    K(\epsilon)\coloneqq (\epsilon-4)e^{\epsilon} -(\epsilon+4)e^{-\epsilon}+2\epsilon^2+8=2\epsilon \sinh \epsilon -8\cosh \epsilon+2\epsilon^2+8.
\end{equation*}
We have:
\begin{align*}
    K'(\epsilon)&=2\epsilon \cosh\epsilon-6\sinh\epsilon+4\epsilon,\\
K''(\epsilon)&=2\epsilon \sinh\epsilon-4\cosh\epsilon+4,\\
K'''(\epsilon)&=2\left(\epsilon \cosh\epsilon-\sinh\epsilon\right).
\end{align*}

It is obvious that $K'''$ is strictly positive on $\mathbb{R}_{>0}$. Since $K''(0)=0$, $K'(0)=0$ and $K(0)=0$ we get successively that $K''$, $K'$ and $K$ are strictly increasing and strictly positive on $\mathbb{R}_{>0}$. Thus, we conclude that for all $\alpha<\beta$, we have $\|Y_{\trunc}\|_{\vp}^2>\text{Var}[Y_{\trunc}]$ so that the truncated exponential is never strictly sub-Gaussian. 


\bibliographystyle{apalike}
\bibliography{references}

\end{document}